\documentclass[a4paper,12pt]{article}

\usepackage{amsthm,amsmath,stmaryrd,bbm,hyperref,geometry,color,authblk}
\usepackage[utf8]{inputenc}
\usepackage{amssymb}
\usepackage[english,french]{babel}
\usepackage{graphicx}
\usepackage{amsfonts,amssymb}
\usepackage{verbatim}
\usepackage{enumitem}
\usepackage{subcaption}

\newcommand{\po}{\left(}
\newcommand{\pf}{\right)}
\newcommand{\co}{\left[}
\newcommand{\cf}{\right]}
\newcommand{\cco}{\llbracket}
\newcommand{\ccf}{\rrbracket}
\newcommand{\R}{\mathbb R}

\newcommand{\N}{\mathbb N} 
\newcommand{\dd}{\mathrm{d}}
\newcommand{\bX}{\mathbf{X}}
\newcommand{\bY}{\mathbf{Y}}

\newcommand{\bx}{\mathbf{x}}
\newcommand{\by}{\mathbf{y}}

\newcommand{\Id}{\mathrm{Id}}
\newcommand{\na}{\nabla}
\newcommand{\1}{\mathbbm{1}} 

\newcommand{\rmv}[1]{}

\newtheorem{thm}{Theorem}
\newtheorem{assu}{Assumption}
\newtheorem*{assu*}{Assumption}
\newtheorem{lem}[thm]{Lemma}

\newtheorem{cor}[thm]{Corollary}
\newtheorem{prop}[thm]{Proposition}
\newtheorem{rem}{Remark}
\newtheorem{ex}{Example}

\title{Long-time propagation of chaos and exit times for metastable  mean-field particle systems}
\author{Pierre Monmarché}

\begin{document}
 \selectlanguage{english}

\maketitle

\begin{abstract}
Systems of stochastic particles evolving in a multi-well energy landscape and attracted to their barycenter is the prototypical example of mean-field process undergoing phase transitions: at low temperature, the corresponding mean-field deterministic limit has several stationary solutions, and the empirical measure of the particle system is then expected to be a metastable process in the space of probability measures, exhibiting rare transitions between the vicinity of these stationary solutions.  We show two results in this direction: first, the exit time from such metastable domains occurs at time exponentially large with the number of particles and follows approximately an exponential distribution; second, up to the expected exit time, the joint law of particles remain close to the law of independent non-linear McKean-Vlasov processes.
\end{abstract}
\selectlanguage{french}
\begin{abstract}
      Les systèmes stochastiques  de particules évoluant dans des paysages d'énergie vallonnés et attirées par leur barycentre constituent un exemple prototypique de processus champ moyen sujets à des transitions de phase: à basse température, la limite champ moyen déterministe a plusieurs solutions stationnaires, et l'on s'attend alors à observer pour  la mesure empirique du système de particules  une dynamique métastable sur l'espace des mesures probabilités, avec des transitions rares entre les voisinages de ces solutions stationnaires. Deux résultats dans cette direction sont établis dans cet article : d'abord, le temps de sortie de tels domaines métastables survient à des temps exponentiellement long en fonction du nombre de particules, et suit approximativement une loi exponentielle. Ensuite, jusqu'au temps de sortie moyen, la loi jointe des particules reste proche de la loi de processus de McKean-Vlasov indépendants.
\end{abstract}
 \selectlanguage{english}

\section{Overview}

Consider a system of $N$ interacting particles $\bX_t = (X_t^1,\dots,X_t^N)$ on $\R^d$ solving
\begin{equation}
\label{eq:particules}
\forall i\in\cco 1,N\ccf\,,\qquad \dd X_t^i = -\na V(X_t^i) \dd t - \frac{1}{N}\sum_{j=1}^N \na W\po X_t^i - X_t^j\pf  \dd t + \sqrt{2}\sigma \dd B_t^i\,,
\end{equation}
where $V,W\in\mathcal C^2(\R^d,\R)$ with $W$ even, $\sigma>0$ and $B^1,\dots,B^N$ are independent Brownian motions on $\R^d$. This process has been the topic of a continuous research activity over more than the last fifty years \cite{mckean1967propagation,ellis1978limit,dawson1986large,LEONARD1987215,Malrieu,carrillo2003kinetic,tugaut2014phase,Pavliotis}. It is the prototypical example of mean-field interacting particle systems, which are met in a wide variety of fields \cite{Chaintron1,Chaintron2} with notably a particularly high activity recently around optimization and machine learning, see \cite{chizat,Szpruch,mei2018mean,mei2019mean} and references within. In this work we are mostly interested in the case $W(x)=\kappa|x|^2/2$ with $\kappa>0$, so that particles are attracted by their center of mass (in fact, without loss of generality by a change of variable, we will focus on $\kappa=1$).

\paragraph{Classical convergence results.} Taking i.i.d. initial conditions $X_0^1,\dots,X_0^N$ distributed according to some initial distribution $\rho_0$ with finite second moment, under suitable conditions on $V,W$, propagation of chaos occurs over fixed time intervals. This means that the empirical distribution of the particles,
\[\pi(\bX_t) = \frac{1}{N}\sum_{i=1}^N \delta_{X_t^i}\,,\]
almost surely converges weakly to the solution of the non-linear Fokker--Planck equation
\begin{equation}
\label{eq:granularmedia}
\partial_t \rho_t = \na\cdot \po (\na V + \rho_t \star \na W) \rho_t \pf + \sigma^2 \Delta\rho_t \,.
\end{equation}
Moreover, writing $\rho_t^N$ and $\rho_t^{k,N}$ respectively the laws of $\bX_t$ and $(X_t^1,\dots,X_t^k)$ for a given $k\geqslant 1$,
%
\begin{equation}
\label{eq:PoCTV}
\sup_{t\in[0,T]} \|  \rho_t^{k,N} - \rho_t^{\otimes k} \|_{TV}  \underset{N\rightarrow \infty}\longrightarrow 0\,
\end{equation}
where $\|\cdot\|_{TV}$ stands for the total variation norm.

Alternatively, for a fixed $N$, \eqref{eq:particules} can be written as
\begin{equation}
\label{eq:diffbX}
\dd \bX_t = -\na U_N(\bX_t^i) \dd t + \sqrt{2}\sigma \dd \mathbf{B}_t\
\end{equation}
where $\mathbf{B}=(B^1,\dots,B^N)$ and
\begin{equation}
\label{eq:UN}
U_N(\bx) = \sum_{i=1}^N V(x_i) + \frac{1}{2N} \sum_{i,j=1}^N W(x_i-x_j)\,.
\end{equation}
In other words, $\bX$ is an overdamepd Langevin diffusion associated to the potential $U_N$. Under mild conditions on $V,W$, it is ergodic with respect to the Gibbs measure
\begin{equation}
\label{eq:Gibbs}
\rho_\infty^N = \frac{e^{-\frac{1}{\sigma^2} U_N(\bx)}}{\int_{\R^{dN}} e^{- \frac{1}{\sigma^2}U_N}} \dd \bx\,.
\end{equation}
Assuming for instance that $V$ is the sum of a strongly convex function and $W$ is convex, classical arguments show that $\rho_\infty^N$ satisfies a log-Sobolev inequality (LSI), meaning that there exists a constant $\lambda_N>0$ such that
\begin{equation}
\label{eq:LSI}
\forall \nu \ll \rho_\infty^N\,,\qquad  \mathcal H(\nu |\rho_\infty^N) \leqslant \frac1{\lambda_N} \mathcal I(\nu |\rho_\infty^N)\,,
\end{equation}
where $\mathcal H$ and $\mathcal I$ stand respectively for the relative entropy and Fisher Information, given by
\[\mathcal H\po \nu|\mu\pf = \int_{\R^{p}}  \ln \po \frac{\dd \nu }{\dd \mu} \pf \dd \nu\,,\qquad \mathcal I\po \nu|\mu\pf = \int_{\R^{p}} \left|\na \ln \frac{\dd \nu}{\dd \mu} \right|^2 \dd \nu \]
for probability measures $\nu,\mu$ on $\R^p$ if $\nu\ll \mu$ and by $+\infty$ otherwise. Denoting by $(P_t^N)_{t\geqslant 0}$ the Markov semi-group associated to~\eqref{eq:diffbX}, the LSI~\eqref{eq:LSI} is equivalent to the exponential decay
 \begin{equation}
\label{eq:entropydecay}
\forall \nu \ll \rho_\infty^N\,,\, t\geqslant 0\,, \qquad  \mathcal H(\nu P_t^N |\rho_\infty^N) \leqslant e^{-\lambda_N t } \mathcal H(\nu |\rho_\infty^N)\,.
\end{equation}
In fact, as shown in \cite{RocknerWang}, starting from any initial distribution $\nu$ with finite second moment, the relative entropy becomes finite for any positive time, with for all $t>0$ a constant $C_t$ (independent from $N$) such that
\[\mathcal{H}(\nu P_t|\rho_\infty^N) \leqslant C_t \mathcal W_2^2 (\nu,\rho_\infty^N)\,.\]
Combined with \eqref{eq:entropydecay} and Pinsker's inequality, this shows that for any initial distribution $\rho_0^{N}\in\mathcal P_2(\R^{dN})$, 
\begin{equation}
\label{eq:CVTVt}
\| \rho_t^{k,N} - \rho_\infty^{k,N}\|_{TV}  \underset{t\rightarrow \infty}\longrightarrow 0\,, 
\end{equation}
with $\rho_\infty^{k,N}$ the first $kd$-dimensional marginal of $\rho_\infty^N$. 

\paragraph{Metastability.} In some cases, for instance if $V$ is strongly convex and $W$ is convex \cite{Malrieu}, the convergence~\eqref{eq:PoCTV} is uniform in $t$ and the convergence~\eqref{eq:CVTVt} is uniform in $N$, namely, starting with i.i.d. initial conditions distributed according to  $\rho_0 \in \mathcal P_2(\R^d)$, taking for instance $k=1$,
 \begin{equation}
\label{eq:CVunif}
\sup_{N\geqslant 1} \| \rho_t^{1,N} - \rho_\infty^{1,N}\|_{TV}  \underset{t\rightarrow 0}\longrightarrow 0\,,\qquad \sup_{t\geqslant 0} \|  \rho_t^{1,N} - \rho_t \|_{TV}  \underset{N\rightarrow \infty}\longrightarrow 0\,.
\end{equation}
In these cases the non-linear flow~\eqref{eq:granularmedia} admits a unique stationary solution. We are interested in the cases where this fails. In this introduction it is sufficient to focus on the emblematic symmetric double-well case, where $d=1$  and
\begin{equation}
\label{eq:doublewell}
V(x)= \frac{x^4}{4} - \frac{x^2}{2}\,.
\end{equation} 
The case~\eqref{eq:doublewell} is a classical example of mean-field particle system exhibiting a phase transition \cite{tugaut2014phase}: there exists a critical temperature $\sigma_c>0$ such that \eqref{eq:granularmedia} admits a unique stationary solution $\mu_0$ (centered) if $\sigma \geqslant \sigma_c$ and three solutions otherwise (one centered, $\mu_0$ and two non-centered, $\mu_+$ and $\mu_-$, symmetric one from the other). In this second situation, the limits as $t\rightarrow \infty$ and $N\rightarrow \infty$  do not commute and \eqref{eq:CVunif} cannot hold.

In fact, at first order the empirical distribution $\pi(\bX_t)$ is approximately a solution of a perturbation of \eqref{eq:granularmedia} by an (infinite-dimensional) Gaussian noise of magnitude of order $1/\sqrt{N}$ \cite{tanaka1981central,FERNANDEZ199733}. By analogy with the classical Freidlin-Wentzell theory for perturbed dynamical systems in finite dimension \cite{freidlin1998random}, the process is expected to be metastable, meaning that $\pi(\bX_t)$ converges quickly (i.e. at a rate independent from $N$) toward one of the two stationary solutions $\mu_+$ and $\mu_-$ (while $\mu_0$ is unstable), and then remains in its vicinity for times of order $e^{cN}$ for some $c>0$, before some abrupt  rare fluctuation of the Brownian motion $\mathbf{B}$, occurring at a time following approximately an exponential distribution (i.e. memoryless and unpredictable) induces a transition to the vicinity of the other stationary solution (and then repeating this behavior). See Section~\ref{sec:illustration} for illustrations. For the mean-field Curie-Weiss model, which is similar to \eqref{eq:particules} except that the spins $X_t^i$ take value in $\{-1,1\}$, the mean magnetization $\frac1N\sum_{i=1}^N X_t^i$ is an autonomous Markov chain in $[-1,1]$, to which the Freidlin-Wentzell theory applies, and thus this metastable behavior is well understood (see the recent \cite{JournelLeBris} and references within). This is not the case for \eqref{eq:particules}. The classical theory does not apply in this infinite-dimensional framework, and it cannot be applied either to the finite-dimensional system $\bX$ since in that case the dimension grows with $N$. Although this question has received much attention over time (see e.g. \cite{GVALANI2020108720,delgadino2021diffusive,carrillo2020long} and references within for recent works), establishing the analogue of most of the results from the Freidlin-Wentzell theory is still an open question (notably, \cite[Theorem 4]{dawson1986large} states a so-called Arhenius or Eyring-Kramers law for the transition times, but to our knowledge it has never been proven). This is the general topic of the present work.

\paragraph{Our approach : a modified process.} In non-metastable cases, a critical fact in the analysis is that the log-Sobolev constant $\lambda_N$ in \eqref{eq:LSI} can be bounded independently from $N$. Dividing \eqref{eq:LSI} by $N$ and passing to the limit $N\rightarrow \infty$, this implies a global non-linear LSI (with the denomination of~\cite{MonmarcheReygner}), which can be interpreted as a Polyak-Łojasiewicz inequality for the free energy associated to~\eqref{eq:granularmedia}. In metastable situations, such an inequality cannot hold globally and the LSI constant $\lambda_N$ in \eqref{eq:LSI} vanishes as $N$ grows \cite{Pavliotis}. Nevertheless,  as shown in~\cite{MonmarcheReygner} a non-linear LSI can still hold locally, namely in some Wasserstein balls centered at some stationary solutions of~\eqref{eq:granularmedia}. In some sense, in this work, we establish the $N$-particle analogue of these local non-linear LSI. The idea is simply to work with a modified process solving
\begin{equation}
\label{eq:particules-modif2}
\forall i\in\cco 1,N\ccf\,,\qquad \dd Y_t^i = -\na V(X_t^i) \dd  t- b\po X_t^i,\pi(\bX_t)\pf  \dd t + \sqrt{2}\sigma \dd B_t^i\,,
\end{equation}
with a drift $b$ designed to fulfill two requirements: first, $b(\cdot ,\mu) = - \na W \star \mu $ for all $\mu$ in  some domain $\mathcal D$ of interest, which implies  that $\bX_t = \mathbf{Y}_t$ up to the exit time of $\pi(\bY_t)$ from $\mathcal D$ (which, for a metastable domain, will typically occur at times exponentially large with $N$). Second, the invariant measure of~\eqref{eq:particules-modif2} satisfies a LSI with constant independent from $N$. Thanks to this, we are in the nice non-metastable case for the modified process where the limits~\eqref{eq:CVunif} in $N$ and $t$ commute. We can thus deduce some properties on the modified process which are then transferred to the initial one. Besides, the global non-linear LSI for the modified process, implied by the uniform LSI (cf.~\eqref{eq:GNLLSI} in Theorem~\ref{thm:meanfieldconsequence}), implies in turn the same inequality for the initial process, but only locally on $\mathcal D$ (which is why we see this as the $N$-particle version of \cite{MonmarcheReygner}). The mean-field limit of the modified process necessarily has a unique stationary solution, corresponding to the uniqueness of the solution of a finite-dimensional fixed-point problem (see Figure~\ref{fig:f}). A conjecture is that, under suitable additional assumptions, this condition is sufficient for the uniform LSI (see \cite{Pavliotis}), but in this work we directly prove the uniform LSI for a specific modification.

\paragraph{Main contributions and organization.} In our situations the uniform LSI for the modified process cannot be deduced from existing results, e.g. \cite{GuillinWuZhang,Songbo,M61}. One of the key point of this work, of independent interest, is the identification of a suitable criterion to prove this uniform LSI in these situations. This  is the content of Theorem~\ref{thm:LSIN}\rmv{, inspired by \cite{Dagallier}}. From this, we obtain the following results that describe the metastable behavior of the initial process :
\begin{enumerate}
\item Initially, the process converges fast (i.e. at a rate independent from $N$) to the vicinity of a stationary solution of the mean-field equation, and remains there for times of order $e^{cN}$ for some $c>0$ (this is Theorem~\ref{thm:main1}).
\item Exit times from  such metastable states approximately follow an exponential distribution (see Theorem~\ref{thm:mainExit}).
\end{enumerate}
Apart from the present introduction, the Appendix which gathers technical results and Section~\ref{sec:illustration} devoted to numerical illustrations, the main body of the work is organized as follows:
\begin{itemize}
\item Section~\ref{sec:quadra} is concerned with the process~\eqref{eq:particules} with a quadratic interaction $W(x) = \frac{1}{2}|x|^2$, which is our initial process of interest. The main results, Theorems~\ref{thm:main1} and \ref{thm:mainExit}, are stated and proven. The proofs remain at high level and heavily rely  on intermediary results postponed to the rest of the work.
\item Sections~\ref{sec:UnifLSI}, \ref{sec:consequences} and \ref{sec:exponential} are \emph{independent} from Section~\ref{sec:quadra}, in particular in terms of notations.  These sections are concerned with a mean-field energy $U_N$ of the form~\eqref{eq:UNgeneral}, more general than the quadratic case of Section~\ref{sec:quadra} (although the particles still interact only through their barycenter). In these sections,  the temperature is normalized to $\sigma^2=1$ for notational convenience.  Section~\ref{sec:UnifLSI} is devoted to the proof of Theorem~\ref{thm:LSIN}, which gives a criterion for a uniform LSI for the Gibbs measure associated to $U_N$. Sections~\ref{sec:consequences} and \ref{sec:exponential} then state some useful consequences of this uniform LSI, respectively on the long-time behavior of the particle system and its mean-field limit, and on the exit times from stable domains.
\item  Finally, Section~\ref{sec:modified} connects the two previous parts. More specifically, starting from the initial process with quadratic interaction, we design in Propositions~\ref{prop:hmodifdim1} (in dimension 1) and \ref{prop:hmodifdimD} (in general) a modified energy of the form~\eqref{eq:UNgeneral} which coincides with the initial one on a given metastable domain but to which the results of Sections~\ref{sec:UnifLSI}, \ref{sec:consequences} and \ref{sec:exponential} apply,  and which is a critical element of the proofs of Theorems~\ref{thm:main1} and \ref{thm:mainExit}.
\end{itemize}

\subsection*{General Notations}

We denote $\mathcal P_2(\R^d)$ and $\mathcal W_2$ respectively the set of probability measures on $\R^d$ with finite second moment and the associated $L^2$ Wasserstein distance. When a measure $\mu\in\mathcal P_2(\R^d)$ has a Lebesgue density, we use the same letter $\mu$ for the density, and we write $\mu \propto g$ if $\mu(x) = g(x)/\int_{\R^d} g$. For $x,y\in\R^d$, $|x|$ and $x\cdot y$ stands for the Euclidean norm and scalar product. For a matrix $A$, $|A|$ stands for the operator norm associated to the Euclidean norm. The closed Euclidean ball of $\R^d$ centered at $x\in\R^d$ with radius $r\geqslant 0$ is denoted $\mathcal B(x,r)$. Similarly, $\mathcal B_{\mathcal W_2}(\nu,r) = \{\mu \in\mathcal P_2(\R^d),\mathcal W_2(\nu,\mu) \leqslant r\}$. For a function $\varphi\in\mathcal C^1(\R^d,\R^n)$ we write $\na \varphi = (\na_{x_i}\varphi_j)_{i \in \cco 1,d\ccf,j\in\cco 1,n\ccf}$ where $i$ stands for the line and $j$ for the column. With this notation, the chain rule reads $\na (\psi\circ \varphi) = \na\varphi \na \psi$. 

\section{Main results with a quadratic interaction}\label{sec:quadra}

\subsection{Settings}

\begin{assu}\label{assu:main-result}
We consider $m_*\in\R^d$, $V,W\in\mathcal C^2(\R^d,\R)$, $\mathcal D\subsetneq \R^d$, $\sigma^2>0$, $\rho_0\in\mathcal P_2(\R^d)$ and a random variable $\bX_0 \in\R^{dN}$ with the following conditions.
\begin{itemize}
\item \textbf{(Confining potential)} we can decompose $V=V_c+V_b$ where $V_c$ is strongly convex and $V_b$ is bounded.  There exist $\beta>2$ and $R_V,\theta>0$ such that $x\cdot\na V(x) \geqslant |x|^\beta$ and $|\na V(x)|\leqslant |x|^\theta$ for all $x\in\R^d$ with $|x|\geqslant R_V$. Finally, $V$ is one-sided Lipschitz continuous, meaning that 
\begin{equation}
\label{eq:one-sided}
\exists L>0,\quad \forall x,y\in\R^d,\qquad (\na V(x)-\na V(y)) \cdot (x-y) \geqslant - L|x-y|^2\,.
\end{equation} 
\item \textbf{(Interaction potential)} For all $x\in\R^d$, $W(x)= \frac{1}{2}|x|^2$.
\item \textbf{(Attractor)} $m_*\in\R^d$ is a fixed-point of $f:\R^d\rightarrow \R^d$ given by
\begin{equation}
\label{eq:def-f}
f(m) = \int_{\R^d} x \nu_{m}(x)\dd x\,,\quad\text{where}\quad \nu_m(x) = \frac{\exp\po  - \frac{1}{\sigma^2}\co V(x) + \frac{1}{2}|x-m|^2\cf \pf}{\int_{\R^d}\exp\po  - \frac{1}{\sigma^2}\co V(y) + \frac{1}{2}|y-m|^2\cf \pf\dd y}\,.
\end{equation} 
Moreover, $|\na f(m_*)|<1$. We write $\rho_*=\nu_{m_*}$.
\item \textbf{(Metastable domain)} $m_*\in\mathcal D$ and there is no other fixed point of $f$ in  $\mathcal D$. Moreover, one of the two following conditions hold:
\begin{itemize}
\item $d=1$ and $\mathcal D$ is a closed interval.
\item for some $r>0$, $\mathcal D = \mathcal B(m_*,r)$, and $|\na f|<1$ over $\mathcal D$.
\end{itemize}
\item \textbf{(Initial distribution)} The solution $(\rho_t)_{t\geqslant 0}$ of \eqref{eq:granularmedia} with initial condition $\rho_0$ converges weakly to $\rho_*$ as $t\rightarrow \infty$. Moreover, $\int_{\R^d} |x|^{\ell} \rho_0(\dd x) <\infty$ with $\ell  = \max(7,  16(\theta-1)^{4})$, and $\bX_0 \sim \rho_0^{\otimes N}$. 
\end{itemize}
\end{assu}

These  conditions ensure that for all $N\geqslant 1$, $\int_{\R^{dN}} e^{-\frac{1}{\sigma^2} U_N} <\infty$ where $U_N$ is defined in~\eqref{eq:UN}, so that the Gibbs measure~\eqref{eq:Gibbs} is well-defined. Given $N\geqslant 1$ and $(\bX_t)_{t\geqslant 0}$ a solution of~\eqref{eq:diffbX} with $\bX_0\sim\rho_0^{\otimes N}$, we write $\rho_t^N$ the law of $\bX_t$,
\[\bar X_t = \frac1N\sum_{i=1}^N X_{t}^i\,,\qquad  \tau_N = \inf\left\{ t\geqslant 0,\ \bar X_t \notin \mathcal D\right\}\,.\]
Under Assumption~\ref{assu:main-result}, whatever the initial distribution, $\tau_N$ is almost surely finite with finite expectation \cite{brassesco1998couplings}.  We set
\begin{equation}
\label{eq:deftN}
t_N = \mathbb E_{\rho_*^{\otimes N}}\po \tau_N\pf\,.
\end{equation}

\begin{ex}
For $d=1$, consider   the double-well potential~\eqref{eq:doublewell} with $\sigma^2 < \sigma_c^2$. It is known that $m_*$ the positive fixed point of $f$ given by~\eqref{eq:def-f} is such that $f'(m_*)<1$ \cite{MonmarcheReygner}. Take $\mathcal D=[\varepsilon,\infty[$ for any $\varepsilon\in(0,m_*)$. Then Assumption~\ref{assu:main-result} holds, given any initial distribution $\rho_0$ satisfying the last condition.
\end{ex}

\begin{ex}
Consider   $V(x)=|x|^4 + V_0(x)$ for some $V_0\in\mathcal C^2(\R^d,\R)$ with bounded Hessian. Let $z_0$ be a non-degenerate local minimizer of $V$ which is the global minimizer of $z\mapsto V(z) + \frac{1}{2}|z-z_0|^2$. Then  there exists a family $(m_{*,\sigma})_{\sigma \in \R_+}$ in $\R^d$ converging to $z_0$ as $\sigma\rightarrow 0$ such that $m_{*,\sigma}$ is a  fixed point  of the corresponding $f$ in \eqref{eq:def-f} for all $\sigma$. Moreover, for $\sigma$ small enough, $|\na f(m_{*,\sigma})|<1$, see \cite{MonmarcheReygner}. Assume that it is the case and take $\mathcal D=\mathcal B(m_{*,\sigma},\varepsilon)$ with $\varepsilon$ small enough so that $|\na f|<1$ over this ball. Then, Assumption~\ref{assu:main-result} holds, given any initial distribution $\rho_0$ satisfying the last condition.
\end{ex}

\subsection{Results}

We start with the following statement:
\begin{prop}\label{prop:tN}
Under Assumption~\ref{assu:main-result}, there exist $a,\eta>0$ such that, for all $N \geqslant 1$, $t_N \geqslant \eta e^{a N}$. 
\end{prop}

More precisely, in this result, we may take $a= \mathcal I-\varepsilon$ for any $\varepsilon>0$ where $\mathcal I$ is the large deviation rate of Dawson and Gärtner, see Remark~\ref{rem:DawsonGartner}.

With this, our first main result states that the particle system quickly converges to the vicinity of $\rho_*$, and remains there for times which are exponentially long with $N$:
\begin{thm}\label{thm:main1}
Under Assumption~\ref{assu:main-result}, there exist $\lambda>0$ (independent from $\rho_0$) and $C_0>0$ such that for all $t\in[1,t_N/\ln N^2]$, 
\begin{equation}
\label{eq:mainresult1}
\mathcal H \po \rho_t^N | \rho_*^{\otimes N}\pf \leqslant C_0 \po N e^{-\lambda t} + 1\pf \,.
\end{equation}
\end{thm}

The second result shows that, for large $N$, for an initial condition $\rho_0$ in the basin of attraction of $\rho_*$ along~\eqref{eq:granularmedia},  the  exit time from $\mathcal D$ approximately follows an exponential distribution, with the same rate as for $\rho_*$.

\begin{thm}\label{thm:mainExit}
Under Assumption~\ref{assu:main-result}, assume furthermore that $m_{\rho_t}$ is in the interior of $\mathcal D$ for all $t\geqslant 0$. Then, for any $k\geqslant 2$, provided $\int_{\R^d} |x|^{\ell}\rho_0(\dd x)<\infty$  with $\ell = \max(2k+3,  2^{2k}(\theta-1)^{2k})$, there exists $C>0$ such that  for all $N\geqslant 1$,
\[\left| \frac{1 }{t_N} \mathbb E_{\rho_0^{\otimes N}}\po \tau_N\pf -1\right| + \sup_{s\geqslant 0} |\mathbb P_{\rho_0^{\otimes N}} \po \tau_N \geqslant s t_N\pf - e^{-s}| \leqslant \frac{C}{N^k}\,.\]
\end{thm}

These results call for a few comments.
\begin{itemize}
\item In Theorem~\ref{thm:main1}, the restriction that $t\geqslant 1$ can  be removed if we assume additionally that $\mathcal H(\rho_0|\rho_*) < \infty$. 
\item The novelty in Theorem~\ref{thm:main1}  with respect to previous finite-time results is that it holds up to time which are exponential with $N$ in metastable situations where uniform-in-time propagation of chaos fails (while standard arguments, involving the Gronwall Lemma, give propagation of chaos up to times of order $\ln N$). Moreover, by comparison, \cite[Proposition 21]{MonmarcheReygner} only states a quick initial decay of $\mathcal H(\rho_t^N|\rho_\infty^N)$ up to the level of $\mathcal H(\rho_*^{\otimes N}|\rho_\infty^N)$, which does not mean that $\rho_*^{\otimes N}$ stays close to $\rho_\infty^{\otimes N}$ and do not provide insight on the timescale at which transitions occur.
\item From the global bound~\eqref{eq:mainresult1} we can state the convergence of the system toward $\rho_*$ in different ways.  Using the sub-additivity property of the relative entropy with respect to tensorized target (see e.g. \cite[Lemma 5.1]{Chen1}) we get that for all $k\in\cco 1,N\ccf$ and $t\in[1,t_N/\ln N^2]$,
\begin{equation}
\label{eq:loc89}
\|\rho_t^{k,N} - \rho_*^{\otimes k}\|_{TV}^2 + \mathcal W_2^2 \po  \rho_t^{k,N} , \rho_*^{\otimes k}\pf + \mathcal H \po \rho_t^{k,N} | \rho_*^{\otimes k}\pf \leqslant C_0' k \po e^{-\lambda t}  + \frac1N\pf \,,
\end{equation}
for some $C_0'>0$ independent from $N$ and $k$, where $\rho_t^{k,N}$ denotes the law of $(X_t^1,\dots,X_t^k)$. Here we used Pinsker and Talagrand inequalities (the latter being implied, thanks to \cite{OttoVillani}, by the uniform-in $N$ log-Sobolev inequality satisfied by $\rho_*^{\otimes N}$ under Assumption~\ref{assu:main-result}, following standard arguments, see for instance the proof of Lemma~\ref{lem:LSImicro}). For an observable $\varphi = \varphi_b+ \varphi_{L}:\R^d \rightarrow \R$  with $\varphi_b$ bounded an $\varphi_L$ Lipschitz continuous, expanding the square and using~\eqref{eq:loc89} with $k=2$ gives 
\begin{equation}
\label{loc:phi67}
\mathbb E \po \left|\frac{1}{N}\sum_{i=1}^N   \varphi(X_t^i) - \int_{\R^d} \varphi\rho_* \right|^2\pf  \leqslant C\po e^{-\lambda t} + \frac{1}{N}\pf\,,
\end{equation}
for some $C>0$ independent from $N$ and $t\in[1,t_N/\ln N^2]$. Finally, considering $\bX_t,\bY_t$ an optimal $\mathcal W_2$-coupling of $\rho_t^N$ and $\rho_*^{\otimes N}$, we get that, for all $N$ and $t\in[1,t_N/\ln N^2]$,
\[\mathbb E \po \mathcal W_2^2 \po \pi(\bX_t),\rho_*\pf\pf  \leqslant \frac{2}{N}\mathcal W_2(\rho_t^N ,\rho_*^{\otimes N}) + 2 \mathbb E \po \mathcal W_2^2 \po \pi(\bY_t),\rho_*\pf\pf \leqslant C \po e^{-\lambda t} +  N^{-\min\po \frac{2}{d},\frac12\pf }\pf\,, \]
for some $C$ thanks to~\eqref{eq:loc89} with $k=N$  and \cite[Theorem 1]{FournierGuillin} (with an additional $\ln N$ factor if $d=4$). The rate $N^{-\min\po \frac{2}{d},\frac12\pf }$ is optimal here. In \eqref{eq:loc89} and \eqref{loc:phi67}, we can expect the optimal rate to be $N^{-2}$ in view of the recent work of Lacker~\cite{Lacker} and successive results \cite{LackerLeFlem,MonmarcheRenWang,RenSongboSize}. This cannot be deduced simply by a global estimate~\eqref{eq:mainresult1} but requires to work with the hierarchy of entropies $\mathcal H ( \rho_t^{k,N} | \rho_*^{\otimes k})$ for all $k\in\cco 1,N\ccf$ directly.

In Theorem~\ref{thm:main1} we stated the convergence of $\rho_t^N$ to $\rho_*^{\otimes N}$ for simplicity. Alternatively, by combining~\eqref{eq:mainresult1} for times $t\in[\delta \ln N,t_N/\ln N^2]$ for some small $\delta>0$ with classical finite-time propagation of chaos estimates, we can also get a control of $\|\rho_t^{k,N}-\rho_t^{\otimes k}\|_{TV} + \|\rho_t^{k,N}-\rho_t^{\otimes k}\|_{TV}$ uniformly over $t\in[0,t_N/\ln N^2]$. This shows that, under Assumption~\ref{assu:main-result}, propagation of chaos occurs up to times which are exponentially large with $N$. See \cite[Section 4]{DelarueTse} for a result in this spirit in a different context.
\item From Theorem~\ref{thm:mainExit} and Proposition~\ref{prop:tN} we also get that the optimal constant $\lambda_N$ in the LSI \eqref{eq:LSI} of $\rho_\infty^N$ goes exponentially fast to $0$ with $N$ (to be compared with the $1/N$ rate of \cite[Theorem 3.3]{Pavliotis}). For instance, in the symmetric double well case in dimension $d=1$, it is clear that $\rho_\infty^N$ is invariant by the symmetry $\bx \mapsto -\bx$. As a consequence, under $\rho_{\infty}^N$, $\mathcal A=\{\frac1N\sum_{i=1}^N X^i \leqslant 0\}$ has probability $1/2$. However, taking $\mathcal D=[\varepsilon,\infty)$ in Theorem~\ref{thm:mainExit} for an arbitrarily small $\varepsilon>0$ and $\rho_0 = \mu_+$,
\[\rho_t^N(\mathcal A) \leqslant \mathbb P_{\mu_+^{\otimes N}}\po \tau_N \leqslant t\pf \leqslant \frac14 \]
for $t = \ln(5) t_N $ for $N$ large enough. By Pinsker inequality, this implies that, for this time $t$,
\[\frac1{32}  \leqslant \mathcal H \po \mu_0^{\otimes N} P_t^N |\rho_\infty^N\pf  \leqslant e^{-\lambda_N t }  \mathcal H \po \mu_0^{\otimes N}  |\rho_\infty^N\pf \,.  \]
Since $ \mathcal H \po \mu_0^{\otimes N}  |\rho_\infty^N\pf$ is of order $N$, we get that $\lambda_N$ is at most of order $\ln(N)/t_N$ as $N\rightarrow \infty$.
\end{itemize}

\subsection{Proofs of the main results}

Here we give the high-level arguments for establishing Proposition~\ref{prop:tN}  and   Theorems~\ref{thm:main1} and \ref{thm:mainExit}, referring to intermediary results stated and proven in the rest of the paper.

We use the notation $a(N,t) \lesssim b(N,t)$ if there exists a constant $C$ independent from $N$ and $t$ such that $a(N,t) \leqslant C b(N,t)$ for all $t\geqslant 0$ and $N\geqslant 1$.

Since $|\na f(m_*)|<1$, \cite[Proposition 9]{MonmarcheReygner} applies, which shows that a local non-linear LSI (see \cite{MonmarcheReygner} or \eqref{eq:GNLLSI}) holds with some constant $\eta>0$ for all measures in $\mathcal A=\{\mu\in\mathcal P_2(\R^d), |m_{\mu} -m_*|\leqslant\varepsilon \}$ for some $\varepsilon>0$. By weak convergence $\rho_t\in \mathcal A$ for $t$ large enough, which by \cite[Theorem 8 and Remark 1]{MonmarcheReygner} shows that
\[\mathcal W_2^2\po \rho_t,\rho_*\pf \lesssim e^{- t/\eta }\,.\]
Fix $r_1>0$ such that $\mathcal B(m_*,r_1) \subset \mathcal D$ and let $r_2\in(0,r_1/2)$ be small enough so that any solution of  \eqref{eq:granularmedia} initialized in   $\mathcal B_{\mathcal W_2}(\rho_*,2r_2)$  remains in  $\mathcal B_{\mathcal W_2}(\rho_*,r_1/2)$ for all times (the existence of $r_2$ is ensured by \cite[Theorem 8]{MonmarcheReygner}). Let $T\geqslant 0$ be such that $\mathcal W_2\po \rho_T,\rho_*\pf \leqslant r_2/2$. Starting at time $T$, we consider $(\bY_t)_{t\geqslant T}$ initialized at $\bY_T=\bX_T$ and solving the modified equation~\eqref{eq:particules-modif2} with
\begin{equation}
\label{loc:b}
b(x,m) = \kappa (x-m) + \na h(m)\,, 
\end{equation}
where $h$ is a  convex $\mathcal C^2$ function with $\|\na h\|_\infty,\|\na^2 h\|_\infty<\infty$, such that $h(m)=0$ for all $m\in\mathcal D$ and the modified Gibbs measure
\begin{equation}
\label{eq:modigiedGibbs}
\tilde \rho_\infty^N \propto \exp\po - \frac{1}{\sigma^2} U_N(\bx) + N   h(\bar x)\pf
\end{equation}
satisfies a uniform-in-$N$ LSI~\eqref{unifLSI} (here, $\bar x=\frac1N\sum_{i=1}^Nx_i$) for some $\lambda>0$. Thanks to Theorem~\ref{thm:LSIN}, the existence of such an $h$ is ensured by Proposition~\ref{prop:hmodifdim1} in  dimension $1$ and Proposition~\ref{prop:hmodifdimD} in general. 

 Set $ \tilde \tau_N := \inf\{t\geqslant 0, \bar Y_{T+t} \notin \mathcal D\}$. If $\tilde\tau_N>0$ then in particular $\bar X_T = \bar Y_T \in \mathcal D$, and then $\bX_t$ and $\bY_t$ coincide for  $t\in[T,T+\tilde\tau_N]$.

\begin{proof}[Proof of Proposition~\ref{prop:tN}]
Since,   on the one hand, $\rho_*$ is a stationary solution for the mean-field limit associated to $\bY$ and, one the other hand, the exit event is the same for the initial and modified dynamics, Proposition~\ref{prop:tN} follows from Corollary~\ref{cor:expoexittimemodif} applied to the modified process with drift~\eqref{loc:b}.
\end{proof}

\begin{proof}[Proof of Theorem~\ref{thm:main1}]
 Using Proposition~\ref{prop:regularize}, for $t\geqslant T$ we bound 
\begin{align}
\mathcal H \po \rho_{t+1}^N |\rho_*^{\otimes N} \pf &\lesssim  \mathcal W_2^2 \po \rho_{t}^N ,\rho_*^{\otimes N} \pf +1 \nonumber\\
&  \lesssim   \mathcal W_2^2 \po \rho_{t}^N ,\tilde \rho_{t}^{ N} \pf + \mathcal W_2^2 \po \tilde \rho_{t}^N ,\rho_*^{\otimes N} \pf  +1\,.\label{loc:H12}
\end{align}
On the one hand, using Theorem~\ref{thm:CVHmain} and the uniform-in-$N$ Talagrand inequality satisfied by $\rho_*^{\otimes N}$, 
\begin{equation}
\label{loc:H13}
\mathcal W_2^2 \po \tilde \rho_{t}^N ,\rho_*^{\otimes N}\pf   \lesssim  e^{-\lambda (t-T)} \mathcal W_2^2 \po  \rho_{T}^N ,\rho_*^{\otimes N}\pf +1 \lesssim  e^{-\lambda t} N + 1\,,
\end{equation}
where we used Lemma~\ref{lem:moments_instantanés} to get that the second moment of $\rho_T^{N}$ is of order $N$. On the second hand,
\begin{multline*}
\mathcal W_2^2 \po \rho_{t}^N ,\tilde \rho_{t}^{ N} \pf \leqslant \mathbb E \po |\bX_t - \bY_t|^2 \pf = N \mathbb E \po |X_t^1-Y_t^1|^2 \1_{\bX_t\neq \bY_t} \pf  \\
\leqslant N \sqrt{\mathbb P \po \bX_t\neq \bY_t \pf \mathbb E \po |X_t^1|^4 + |Y_t^1|^4 \pf   } \lesssim  N \sqrt{\mathbb P_{\rho_T^N} \po \tilde \tau_N \leqslant t\pf  }\,, 
\end{multline*}
using Lemma~\ref{lem:moments_instantanés}. Then
\[\mathbb P_{\rho_T^N} \po \tilde \tau_N \leqslant t\pf \leqslant \mathbb P_{\rho_0^{\otimes N}} \po \mathcal W_2(\pi(\bX_T),\rho_*)>r_2\pf + \sup_{\bx\in\R^{dN},\mathcal W_2(\pi(\bx),\rho_*)\leqslant r_2} \mathbb P_{\bx} \po \tilde \tau_N \leqslant t\pf\,.\]
First, since $\mathcal W_2(\rho_T,\rho_\infty)\leqslant r_2/2$,
\[ \mathbb P_{\rho_0^{\otimes N}} \po \mathcal W_2(\pi(\bX_T),\rho_*)>r_2\pf \leqslant \mathbb P_{\rho_0^{\otimes N}} \po \mathcal W_2(\pi(\bX_T),\rho_T)>r_2/2\pf \lesssim \frac{1}{N^2}\]
thanks to Proposition~\ref{prop:finiteTimePOC} (applied with $k=2$).   Second, thanks to Theorem~\ref{thm:exitdanslecasLSI},
\[ \sup\left\{ \mathbb P_{\bx} \po \tilde \tau_N \leqslant t\pf,\ \bx\in\R^{dN} \text{ with }\mathcal W_2(\pi(\bx),\rho_*)\leqslant r_2,\ t\leqslant t_N/\ln N^2\right\} \lesssim e^{- \ln N^2} + \frac{1}{N^2}\,. \]
We have thus obtained that
\[  \sup_{t\leqslant t_N/\ln N^2} \mathcal W_2^2 \po \rho_{t}^N ,\tilde \rho_{t}^{ N} \pf  \lesssim 1\,. \] 
Combining this with~\eqref{loc:H12} and \eqref{loc:H13}, at this stage, we have shown that there exists $C_0>0$ independent from $N$ and $t$ such that~\eqref{eq:mainresult1} holds for all $t\in [T+1,t_N/\ln N^2]$. The fact that the inequality can be extended to $t\in[1,T+1]$ follows from
\begin{multline*}
\sup_{t\in[0,T]}\mathcal H \po \rho_{t+1}^N |\rho_*^{\otimes N} \pf \lesssim  \sup_{t\in[0,T]}\mathcal W_2^2 \po \rho_{t}^N ,\rho_*^{\otimes N} \pf +1\\
 \lesssim  \sup_{t\in[0,T]}\co \mathcal W_2^2 \po \rho_{t}^N ,\rho_t^{\otimes N} \pf + \mathcal W_2^2 \po \rho_{t}^{\otimes N} ,\rho_*^{\otimes N} \pf\cf  +1 \lesssim N
\end{multline*}
by classical finite-time propagation of chaos results. This concludes the proof of Theorem~\ref{thm:main1}.

\end{proof}

\begin{proof}[Proof of Theorem~\ref{thm:mainExit}]  As discussed above, if $\tau_N >T$ then $\tau_N = T + \tilde \tau_N$. Hence
\begin{align*}
|\mathbb P_{\rho_0^N} \po \tau_N \geqslant s t_N\pf - e^{-s}| & \leqslant \mathbb P_{\rho_T^N} \po \tilde \tau_N \geqslant s t_N -T\pf - e^{-s}|  + \mathbb P_{\rho_0^N}(\tau_N \leqslant T) \\
&\leqslant \sup_{s'\geqslant 0} |\mathbb P_{\rho_T^N} \po \tilde \tau_N \geqslant s' t_N\pf - e^{-s'}|  +  \left| e^{T/t_N} - 1\right| +  \mathbb P_{\rho_0^N}(\tau_N \leqslant T)  \\
&\lesssim \frac{1}{N^m }+ \mathbb P_{\rho_0^N}(\mathcal W_2(\pi(\bX_T),\rho_*)\geqslant r_2 ) +   \mathbb P_{\rho_0^N}(\tau_N \leqslant T)
\end{align*}
for any $m\geqslant 1$ thanks to Theorem~\ref{thm:exitdanslecasLSI} and Proposition~\ref{prop:tN} (where we used that  the exit time from $\mathcal D$ is the same for the initial dynamics~\eqref{eq:diffbX} and the modified one~\eqref{eq:particules-modif2}, so that $t_N$ here and in Theorem~\ref{thm:exitdanslecasLSI} applied to the modified problem coincide). The second term is bounded thanks to  Proposition~\ref{prop:finiteTimePOC} since $\mathcal W_2(\rho_T,\rho_*) \leqslant r_2/2$. It only remains to treat the third term. The condition that $m_{\rho_t}$ is in the interior of $\mathcal D$ for all $t\geqslant 0$ shows that $\varepsilon:=\sup_{t\in[0,T]} \mathrm{dist}(m_{\rho_t},\mathcal D^c) >0$. Then
\[\mathbb P( \tau_N \leqslant T ) \leqslant \mathbb P \po \sup_{t\in[0,T]} |\bar X_t - m_{\rho_t}| \geqslant \varepsilon/2\pf \leqslant \mathbb P \po \sup_{t\in[0,T]} \mathcal W_2\po \pi (\bX_t) , \rho_t\pf  \geqslant \varepsilon/2\pf \lesssim \frac{1}{N^k}\]
thanks again to Proposition~\ref{prop:finiteTimePOC}.
\end{proof}

\section{Uniform LSI}\label{sec:UnifLSI}

As mentioned in the introduction, this section is completely independent from the previous ones, in particular as far as notations are concerned.

Given some  $V_0\in\mathcal C^2(\R^d,\R)$ and $h_0\in\mathcal C^2(\R^d,\R)$, consider on $\mathcal P_2(\R^d)$ the mean-field energy
\begin{equation}
\label{eq:E}
\mathcal E(\mu) = \int_{\R} V_0 \dd \mu + h_0(m_\mu)\,,\quad \text{where}\quad m_\mu = \int_{\R^d} x \mu(\dd x)\,.  
\end{equation}
For $N\geqslant 1$ and $\bx\in\R^{dN}$, writing $\bar x= \frac1N\sum_{i=1}^N x_i$, the associated $N$-particle energy is defined as
\begin{equation}
\label{eq:UNgeneral}
 U_N(\bx) = N  \mathcal E \po \pi(\bx)\pf =  \sum_{i=1}^N V_0(x_i) +  N h_0(\bar x) \,. 
\end{equation}
The basic conditions on $V_0$ and $h_0$ are the following.
\begin{assu}\label{assu:LSIN1}
 The potential $V_0\in\mathcal C^2(\R^d,\R)$ and the energy $h_0\in\mathcal C^2(\R^d,\R)$ satisfy:
\begin{enumerate}
\item $V_0=V_c+V_b$ where $V_c$ is strongly convex and $V_c$ is bounded.
\item $h_0(m) = h(m) - \frac{\kappa}{2}|m|^2$ where $\kappa>0$ and $h\in\mathcal C^2(\R^d,\R)$ is a Lipschtiz-continuous convex function with bounded second derivative.
\item For all $N\geqslant 1$, $\int_{\R^{dN}} e^{-U_N} <\infty$. 
\end{enumerate}
\end{assu}
The goal of this section is to establish, under suitable conditions, a  uniform-in-$N$ LSI for the associated Gibbs measure on $\R^{dN}$ with density
\[\rho_\infty^N(\bx) = \frac{e^{-U_N(\bx)}}{\int_{\R^{dN}}e^{- U_N}}\,, \]
that is to say to prove that there exists $\lambda>0$ such that 
\begin{equation}
\label{unifLSI}
\forall N\geqslant 1,\ \forall \rho\in\mathcal{P}_2(\R^{dN})\,,\qquad \mathcal H(\rho|\rho_\infty^N) \leqslant \frac{1}{\lambda}\mathcal I(\rho|\rho_\infty^N)\,.
\end{equation}
Assumption~\ref{assu:LSIN1} is not sufficient to exclude multiple stationary solutions in the mean-field limit and thus a metastable situation for the $N$ particles system (take for instance $h = 0$). However these situations will be prohibited by an additional assumption. To state this last condition, it is convenient to introduce for $\theta\in\R^d$ the free energy  on $\mathcal P_2(\R^d)$ given by 
\[\mathcal F_{\theta}(\mu)  =  \int_{\R^d} (V_0(x) - x\cdot \theta) \mu(\dd x) + h\po m_\mu\pf + \int_{\R^d} \mu \ln \mu  \,, \]
and set
\begin{equation}
\label{eq:defw}
w(\theta) = \frac{|\theta|^2}{2\kappa} + \inf_{\mu\in\mathcal P_2(\R^d)}\mathcal F_{\theta}(\mu) = \frac{|\theta|^2}{2\kappa} + \inf_{\mu\in\mathcal P_2(\R^d)}\left\{   \mathcal F_{0}(\mu) - \theta\cdot m_{\mu}\right \}\,.
\end{equation}
The well-posedness of $w$ under Assumption~\ref{assu:LSIN1} is a consequence of Lemma~\ref{lem:Ftheta} below.

\begin{assu}\label{assu:LSIN2}
There exists $\theta_*\in\R^d$ and $\eta>0$ such that $(\theta-\theta_*)\cdot \na w(\theta) \geqslant \eta|\theta-\theta_*|^2$. Moreover, $w\in\mathcal C^2(\R^d,\R)$ and its Hessian is lower bounded by some (possibly negative) constant.
\end{assu}

This is a weaker condition than requiring that $w$ is strongly convex (indeed, we want to avoid this restriction, see Remark~\ref{rem:convex}). Besides, it implies that $w$ goes to infinity at infinity, hence admits at least a critical point, which is thus unique and equal to $\theta_*$.

\begin{thm}\label{thm:LSIN}
Under Assumptions~\ref{assu:LSIN1} and~\ref{assu:LSIN2}, there exists $\lambda>0$ such that $\rho_\infty^N$ satisfies a LSI with constant $1/\lambda$  for all $N\geqslant 1$. 
\end{thm}

To establish this result, we first decompose $\rho_\infty^N$ as in \cite{bauerschmidt2019very}. Considering  $\Theta\sim \mathcal N(0,\frac{\kappa}{N} I_d)$, applying that 
\[\exp\po \frac{\kappa |z|^2}{2N}  \pf  = \mathbb E \po  e^{z\cdot \Theta }\pf \]
with $z=\sum_{i=1}^N x_i\in \R^d$, the expectation of an observable $\varphi$ with respect to $ \rho_\infty^N$ can be written as
\begin{align*}
\int_{\R^{dN}} \varphi (\bx)  \rho_\infty^N(\dd \bx) &\propto  \mathbb E \po \int_{\R^{dN}} \varphi (\bx) \exp \po -   \sum_{i=1}^N \co V_0(x_i) -   x_i\cdot \Theta \cf  - N h\po\bar x\pf  \pf  \dd \bx \pf \\
& = \int_{\R^d } \int_{\R^{dN}} \varphi(\bx) \mu_{\theta}^{ N}(\bx) \dd \bx \nu_N(\theta)\dd \theta   
\end{align*}
where $\mu_{\theta}^{ N}$ and $\nu_N$ are probability densities given by
\[\mu_{\theta}^{ N}(\bx) = \frac{1}{Z_N(\theta)}\exp \po -   \sum_{i=1}^N \co V_0(x_i) -   x_i\cdot \theta \cf  - N h\po\bar x\pf  \pf\]
with the partition function
\[Z_N(\theta) = \int_{\R^{dN}} \exp \po -   \sum_{i=1}^N \co V_0(x_i) -   x_i\cdot \theta \cf  - N h\po\bar x\pf  \pf\dd \bx\]
and 
\[\nu_N(\theta) \propto Z_N(\theta) \exp\po - \frac{N |\theta|^2}{2\kappa}  \pf\,.  \]
Notice that $\mu_\theta^N $ is the $N$-particle Gibbs measure associated to the free energy $\mathcal F_{\theta}$. 

Following the approach of \cite{bauerschmidt2019very}, the two main intermediary steps to establish Theorem~\ref{thm:LSIN} are the following, whose proofs are postponed to the rest of the section.

\begin{lem}\label{lem:LSImicro}
Under Assumption~\ref{assu:LSIN1}, there exists $c_0>0$  such that $\mu_{\theta}^{ N}$ satisfies a LSI with constant $c_0$ for all $N\geqslant 1$ and $\theta\in\R$.
\end{lem}

\begin{lem}\label{lem:LSImacro} 
Under Assumptions~\ref{assu:LSIN1} and~\ref{assu:LSIN2}, there exists $c_1>0$ such that for all $N\geqslant 1$, $\nu_N$ satisfies a LSI with constant $c_1/N$.
 \end{lem}

 \begin{proof}[Proof of Theorem~\ref{thm:LSIN}]
 For any $\varphi\in\mathcal C^1(\R^{dN},\R)$,
 \[\mathrm{Ent}_{\rho_\infty^N}(\varphi^2) = \int_{\R^d} \mathrm{Ent}_{\mu_\theta^N}(\varphi^2) \nu_N(\dd \theta) + \mathrm{Ent}_{\nu_N}(\Phi^2) \]
 with
 \[\Phi^2(\theta) = \mu_\theta^N(\varphi^2)\,. \]
 Using Lemmas~\ref{lem:LSImicro} and \ref{lem:LSImacro},
 \begin{equation}
 \label{eq:conclusion}
 \mathrm{Ent}_{\rho_\infty^N}(\varphi^2)  \leqslant c_0 \rho_\infty^N\po |\na \varphi|^2\pf  + \frac{c_1}{N} \nu_N(|\na \Phi|^2)\,.
 \end{equation}
  Here
   \[\na \Phi(\theta) = \frac{1}{2\Phi(\theta) } \na(\Phi^2)(\theta) =   \frac{1}{2\Phi(\theta) } \mathrm{Cov}_{\mu_\theta^N}(\varphi^2,\psi)   \] 
  with $\psi(\bx) = \sum_{i=1}^N x_i$.   By \cite[Proposition 2.2]{Ledoux}, using the uniform LSI of $\mu_\theta^N$ and that $\psi$ is $\sqrt{N}$-Lipschitz, there exists $C>0$, independent from $\theta$ and $N$, such that  
     \[|\na \Phi(\theta)|^2 \leqslant CN \int_{\R^{dN}} |\na \varphi|^2 \dd \mu_\theta^N\,.    \] 
     Plugging this in~\eqref{eq:conclusion} concludes the proof.
 \end{proof}

\begin{proof}[Proof of Lemma~\ref{lem:LSImicro}]
This is based on \cite[Theorem 1]{Songbo}. For $\theta\in\R^d$, consider on $\mathcal P_2(\R^d)$ the energy
\[F(\mu) = \int_{\R} \po V_0(x) - \theta \cdot  x\pf \mu(\dd x) + h(m_{\mu})\,.\] 
It is convex along flat interpolations $t\mapsto (1-t) \mu_0 + t \mu_1$, since $h$ is convex (this is condition 5 in the assumptions of \cite[Theorem 1]{Songbo}). The flat derivatives of $F$ are
\[\frac{\delta F}{\delta \mu}(\mu,x) = V_0(x) - \theta \cdot x + \na h(m_\mu)\cdot  x\,,\qquad  \frac{\delta^2 F}{\delta^2 \mu}(\mu,x,x') =  x \cdot \na^2 h(m_\mu) x'\,,\]
and its intrinsic derivatives $D F = \na_x\frac{\delta F}{\delta \mu}$ and $D^2 F= \na_{x,x'}^2 \frac{\delta^2 F}{\delta^2 \mu}$ are
\[DF(\mu,x)= \na V_0(x) - \theta + \na h(m_{\mu})\,,\qquad D^2F(\mu,x,x') = \na^2 h(m_\mu)\,.\]
In particular, $D^2F$ is bounded by $\|\na^2 h\|_\infty$ independently from $\theta$ and $\mu$ (this is condition 1 of \cite[Theorem 1]{Songbo}).

From  $V_0 = V_c+V_b$, we can decompose $\frac{\delta F}{\delta \mu}(\mu,\cdot )$ as the sum of $x\mapsto V_c(x) + (  \na h (m_\mu)-\theta )\cdot  x$ which is strongly convex (with a convexity constant independent from $\theta$ and $\mu$) and of $V_b$ which is bounded (independently from $\theta$ and $\mu$). By the Bakry-Emery criterion and the Holley-Stroock perturbation lemma, the density proportional to $\exp(-\frac{\delta F}{\delta \mu}(\mu,\cdot))$ satisfies a LSI with constant independent from $\theta$ and $\mu$ (this is condition 3 of \cite[Theorem 1]{Songbo}).

For $N>1$, fixing the value of $x_2,\dots,x_N$, the conditional density proportional to
\[x_1 \mapsto \mu_{\theta}^{ N}(\bx)\]
is the perturbation of 
\[x_1 \mapsto e^{-V_0(x_1)+\theta \cdot  x_1}\,,\]
which satisfies a LSI with constant independent from $N,\theta$ and $x_2,\dots,x_N$ (again by Bakry-Emery and Holley-Stroock) by the perturbation $x_1 \mapsto N h(\bar x)$, which is $\|\na h\|_\infty$-Lipschitz continuous for all $N,x_1,\dots,x_2$. As a consequence, by the Aida-Shigekawa theorem \cite{Aida}, this conditional density satisfies a LSI with constant independent from $N,x_2,\dots,x_N$ and $\theta$ (this is condition 4 of \cite[Theorem 1]{Songbo}).

The last assumption (condition 2) of \cite[Theorem 1]{Songbo} is not satisfied in our case but as noticed in the proof of \cite[Theorem 2]{M61} (which generalizes \cite[Theorem 1]{Songbo}) this condition is not necessary for the LSI. As a consequence, \cite[Theorem 1]{Songbo} applies and gives a bound on the LSI constant of $\mu_{\theta}^{ N}$ which is independent from $N$ and expressed in terms of the constant introduced in the conditions checked above, which in our case do not depend on $\theta$.

\end{proof}

We now turn to the proof of Lemma~\ref{lem:LSImacro}\rmv{, inspired by \cite{Dagallier} (whose results don't apply directly in our settings as the energy $\mathcal E$ does not correspond to a pairwise interaction when $h$ is not quadratic)}. We will need several preliminary results.

\begin{lem}\label{lem:Ftheta}
Under Assumption~\ref{assu:LSIN1}, for $\theta\in\R$,   the free energy $\mathcal F_{\theta}$ admits a unique global minimizer $\rho_{*,\theta}$. It admits a density which is the unique solution of the self-consistency equation
\begin{equation}
\label{demovariance2}
\rho_{*,\theta}(x) \propto \exp\po - V_0(x) + x\cdot \co \theta -  \na h(y_\theta)   \cf \pf\,,
\end{equation}
where $y_{\theta} := m_{\rho_{*,\theta}}$. Moreover, its mean $y_{\theta}$ is the unique fixed-point  of $f_\theta:\R^d\rightarrow\R^d$ given by
\[ f_\theta(y) = \frac{\int_{\R^d} x \exp \po - V_0(x) + x\cdot \co  \theta - \na h(y)\cf \pf\dd x  }{\int_{\R^d} \exp \po - V_0(x) + x\cdot \co \theta-  \na h(y) \cf \pf  \dd x }\,. \]
Finally, $w$ defined by~\eqref{eq:defw} satisfies, for all $\theta\in\R^d$,
\begin{equation}
\label{eq:reformule_w}
w(\theta) = \frac{|\theta|^2 }{2\kappa } + h(y_\theta) - \na h(y_\theta)\cdot  y_\theta -   \ln \int_{\R^{d}}  \exp \po -   V_0(x) +  x\cdot \co  \theta -  \na h (y_\theta)  \cf  \pf \dd x \,. 
\end{equation}
\end{lem}
\begin{proof}
For any $\theta\in\R^d$, $\mathcal F_{\theta}$ is strictly convex along flat interpolations (because $h$ is convex and the entropy is strictly convex). Hence, it is sufficient to show that $\mathcal F_\theta$ is lower bounded to get, by semi-continuity and convexity as in \cite[Lemma 10.2]{mei2018mean}, existence and uniqueness of a global minimizer. Bounding $h(m) \geqslant h(0) -\|\na h\|_\infty |m_\mu|$ and using that $V_0(x) \geqslant b |x|^2 -C$ for some $b,C>0$ under Assumption~\ref{assu:LSIN1},
\begin{align*}
\mathcal F_\theta(\mu) & \geqslant b\int_{\R^d}|x|^2  \mu(\dd x) -C - \po |\theta|+\|\na h\|_\infty\pf |m_\mu| + \int_{\R^d} \mu\ln\mu \\
& \geqslant b\int_{\R^d}|x|^2  \mu(\dd x) -C - \po |\theta|+\|\na h\|_\infty\pf |m_\mu| + \int_{\R^d} \mu\ln\mu \\
& \geqslant \frac{b}2 \int_{\R^d}|x|^2  \mu(\dd x) -C' + \int_{\R^d} \mu\ln\mu 
\end{align*} 
for some $C'>0$ by Jensen inequality, which we use again to get that 
\[\frac{b}2 \int_{\R^d}|x|^2  \mu(\dd x) + \int_{\R^d} \mu\ln\mu  \geqslant -  \ln \int_{\R^d} e^{-\frac{b}{2}|x|^2} \dd x > -\infty\,.  \]

The other  properties stated in Lemma~\ref{lem:Ftheta} then follows from the analysis in \cite{MonmarcheReygner}, that we briefly recall here for completeness. The self-consistency equation~\eqref{demovariance2} is  \cite[Equation (12)]{MonmarcheReygner}. Denoting by $\Gamma_{\theta}(\mu)$ the probability density
\[\Gamma_\theta(\mu)(x) \propto \exp\po - V_0(x) + x\cdot \co \theta - \na h(m_\mu) \cf \pf\,, \]
we decompose
\begin{equation}
\label{eq:decomposeFtheta}
\mathcal F_\theta(\mu) = \mathcal H \po \mu |\Gamma_\theta(\mu)\pf + g_{\theta}(m_\mu)
\end{equation}
with
\begin{equation}
\label{eq:g}
g_\theta(y) = h(y) - \na h(y)\cdot y - \ln \int_{\R^{d}}  \exp \po -   V_0(x) +  x \cdot \co \theta - \na h(y) \cf\pf \dd x\,.
\end{equation}
As in the proof of Lemma~\ref{lem:LSImicro}, by the Bakry-Emery and Holley-Stroock criteria, $\Gamma_\theta(\mu) $ satisfies a LSI with a constant $C>0$ independent from $\theta$ and $\mu$. Since $h$ is convex,  \cite[Equation (15)]{MonmarcheReygner} then gives
\[\mathcal F_{\theta}(\mu) - \inf\mathcal F_\theta \leqslant \mathcal H(\mu|\Gamma(\mu)) \leqslant C \mathcal I(\mu|\Gamma(\mu))\,,\]
which is referred to in \cite{MonmarcheReygner} as a global non-linear LSI. This implies that any $\mu$ satisfying $\mu = \Gamma(\mu)$ (namely \eqref{demovariance2}) is a global minimizer of $\mathcal F_\theta$, from which $\rho_{*,\theta}$ is the unique solution of \eqref{demovariance2}. Moreover, 
\[y_\theta = \int_{\R^d} x \rho_{*,\theta}(x) \dd x = \int_{\R^d} x \Gamma(\rho_{*,\theta})(x) \dd x = f_\theta(y_\theta)\,.\]
Conversely, if $y$ is a fixed point of $f_\theta$, then by the same computation we see that the probability density proportional to $\exp\po - V_0(x) + x\cdot \co \theta - \na h(u) \cf \pf$ is fixed by $\Gamma$, which means that it is $\rho_{*,\theta}$ and that $y=f(y) = m_{\rho_{*,\theta}} =y_\theta$.

By the definition~\eqref{eq:defw}, the decomposition \eqref{eq:decomposeFtheta} and the self-consistency equation,
\[w(\theta)= \frac{|\theta|^2}{2\kappa} +  \mathcal F_\theta(\rho_{*,\theta}) = \frac{|\theta|^2}{2\kappa}  + g_\theta(y_\theta)\,,  \]
which concludes the proof of the lemma. 
\end{proof}

\begin{lem}\label{lem:PoCstationnaire}
Under Assumption~\ref{assu:LSIN1}, there exists $C>0$ independent from $N$ and $\theta\in\R$ such that 
\[  N \int_{\R^{dN}} |\bar x - y_\theta|^2 \mu_\theta^N(\bx)\dd \bx \leqslant C\,.\]
\end{lem}
\begin{proof}
Writing $\varphi(\bx)=\bar x$, we decompose
\begin{equation}
\label{demovariance1}
\mathbb E_{\mu_\theta^N} \po |\varphi - y_\theta|^2\pf = \mathrm{Var}_{\mu_\theta^N}(\varphi) + \left| \mathbb E_{\mu_\theta^N}(\varphi) - y_\theta\right|^2 \,.
\end{equation}
Applying the  Poincaré inequality (uniform $N$ and $\theta$) for $\mu_\theta^N$ implied by the LSI of Lemma~\ref{lem:LSImicro} gives 
\[ \mathrm{Var}_{\mu_\theta^N}(\varphi) \leqslant \frac{c_0 d }{N}\,. \]
To treat  the second term, we use the self-consistency equation~\eqref{demovariance2} to get, for each $i\in\cco 1,N\ccf$,
\[\na_i \ln \rho_{*,\theta}^{\otimes N}(\bx) - \na_i \ln \mu_\theta^N (\bx) = - \na h(y_\theta) + \na h (\bar x)\,, \]
hence
\begin{align*}
\mathcal I \po \rho_{*,\theta}^{\otimes N}|\mu_\theta^{N} \pf &= \int_{\R^{dN}} \left| \na \ln \frac{\rho_{*,\theta}^{\otimes N}}{\mu_\theta^N }\right|^2 \rho_{*,\theta}^{\otimes N}  \\
&=  N \int_{\R^{dN}} | \na h(y_\theta) - \na h(\bar x) |^2 \rho_{*,\theta}^{\otimes N}(\bx)\dd \bx \\
& \leqslant  N \|\na^2 h\|_\infty^2 \int_{\R^{dN}} | y_\theta - \bar x |^2 \rho_{*,\theta}^{\otimes N}(\bx)\dd\bx  \\
&=\|\na^2 h\|_\infty^2 \mathrm{Var} \po \rho_{*,\theta}\pf \,.
\end{align*}
Reasoning as in the proof of Lemma~\ref{lem:LSImicro} (with Bakry-Emery and Holley-Stroock criteria), we see from~\eqref{demovariance2} that $\rho_{*,\theta}$ satisfies a LSI (hence a Poincaré inequality) with constant independent from $\theta$, from which its variance is bounded uniformly in $\theta$. 

Using the LSI for $\mu_\theta^N$ (uniform in $N$ and $\theta$) and the Talagrand inequality it implies, we get that
\[\mathcal W_2^2 \po \rho_{*,\theta}^{\otimes N}, \mu_\theta^{N} \pf  \leqslant C\]
for some $C>0$ independent from  $N$ and $\theta$. Using the interchangeability of particles,
\[\left| \mathbb E_{\mu_\theta^N}(\varphi) - y_\theta\right|^2 \leqslant \frac1N \mathcal W_2^2 \po \rho_{*,\theta}^{\otimes N}, \mu_\theta^{N} \pf \leqslant \frac{C}{N}\,.\] 
Plugging this bound in~\eqref{demovariance1} concludes the proof. 
\end{proof}

\begin{lem}\label{lem:nuNperturbBorne}
Under Assumption~\ref{assu:LSIN1}, 
\[\nu_N(\theta) \propto e^{-N w(\theta) + R_N(\theta)}  \] 
where $R_N(\theta)$ is bounded uniformly over $N\geqslant 1$ and $\theta\in\R^d$.
\end{lem} 

\begin{proof}
Denoting $\gamma_{N,\theta}(\bx)= \exp \po -   \sum_{i=1}^N \co V(x_i) -  x_i\cdot \theta \cf  - N h\po\bar x\pf \pf $,
\begin{eqnarray*}
\nu_N(\theta) & \propto & \exp\po -\frac{N |\theta|^2}{2\kappa}\pf \int_{\R^{dN}} \gamma_N(\bx) \dd \bx \,.
\end{eqnarray*}
Similarly, thanks to~\eqref{eq:reformule_w}, 
\begin{eqnarray*}
e^{-N w(\theta)}  &  = & \exp\po -\frac{N |\theta|^2}{2\kappa}\pf \int_{\R^{dN}} \tilde \gamma_N(\bx) \dd \bx 
\end{eqnarray*}
with 
\[\tilde \gamma_{N,\theta}(\bx)= \exp \po -   \sum_{i=1}^N \co V(x_i) -  x_i\cdot \theta \cf  - N \co h(y_\theta) + \na h(y_\theta)\cdot (  \bar x -  y_\theta ) \cf  \pf \,.\]
 Thanks to Lemma~\ref{lem:PoCstationnaire} and Markov's inequality, for any $C'\geqslant C$,
\[\mu_\theta^N \po A^c   \pf \leqslant \frac{1}{2}\,,\quad \text{where}\quad A = \{\bx\in\R^N,\ N |y_\theta - \bar x|^2 \leqslant 2C' \}\,.\]
As a consequence,
\[\frac12 \int_{\R^N}\gamma_{N,\theta}(\bx) \dd \bx \leqslant \int_{A}\gamma_{N,\theta}(\bx) \dd \bx  \leqslant \int_{\R^N}\gamma_{N,\theta}(\bx) \dd \bx \,.   \] 
Similar bounds hold  for $\tilde \gamma_{N,\theta}$. Indeed, using the representation~\eqref{demovariance2}, we see that, up to the normalizing constant, $\tilde \gamma_{N,\theta}$ is the density of $N$ i.i.d. variables distributed according to $\rho_{*,\theta}$, whose expectation is by definition $y_\theta$. Then, by Markov's inequality, 
\[\rho_{*,\theta}^{\otimes N} \po A^c \pf \leqslant \frac{\mathrm{Var}(\rho_{*,\theta})}{2C'}\,. \]
We have seen in the proof of Lemma~\ref{lem:PoCstationnaire} that $\mathrm{Var}(\rho_{*,\theta})$ is bounded uniformly in $\theta$. Thus we may take $C'\geqslant C$ large enough (independent from $\theta$ and $N$) such that 
\[\frac12 \int_{\R^N}\tilde \gamma_{N,\theta}(\bx) \dd \bx \leqslant \int_{A}\tilde \gamma_{N,\theta}(\bx) \dd \bx  \leqslant \int_{\R^N}\tilde \gamma_{N,\theta}(\bx) \dd \bx \,.   \]

Now, for $\bx\in A$, 
\[ N| h\po\bar x\pf - h(y_\theta) - \na h(y_\theta)\cdot ( y_\theta - \bar x) | \leqslant  \|\na^ 2 h\|_\infty C'\,, \]
from which  
\[ e^{-  \|\na^ 2 h\|_\infty C'} \leqslant \frac{\int_{A}\tilde \gamma_{N,\theta}(\bx) \dd \bx}{\int_{A}\gamma_{N,\theta}(\bx) \dd \bx} \leqslant e ^{ \|\na^ 2 h\|_\infty C'}\,.\]
As a conclusion, 
\[ \frac12 e^{-  \|\na^ 2 h\|_\infty C'} \leqslant \frac{\int_{\R^N }\tilde \gamma_{N,\theta}(\bx) \dd \bx}{\int_{\R^N}\gamma_{N,\theta}(\bx) \dd \bx} \leqslant 2 e ^{ \|\na^ 2 h\|_\infty C'}\,,\]
which concludes the proof.
\end{proof}

 \begin{proof}[Proof of Lemma~\ref{lem:LSImacro}]
 According to Lemma~\ref{lem:nuNperturbBorne} and Holley-Stroock's perturbation lemma, it is sufficient to prove that $\tilde \nu_N = e^{-N w}/\int_{\R^d} e^{-Nw}$ satisfies a LSI with constant of order $1/N$. This would be a consequence of the Bakry-Emery criterion if $w$ were convex, but we wanted to avoid this assumption (as explained in Remark~\ref{rem:convex}). Rather, it follows from Assumption~\ref{assu:LSIN2} and Proposition~\ref{prop:LSIcontractif}.
 \end{proof}

\section{Consequences of the uniform LSI}\label{sec:consequences}

In the same  general settings (and with the same notations) as in Section~\ref{sec:UnifLSI}, in this section we list some useful implications of the uniform LSI~\eqref{unifLSI} in terms of long-time behavior of the particle system and its mean-field limit. They are well-known for pairwise interactions, see for instance~\cite{GuillinWuZhang,Pavliotis}, and hold in greater generality than the restrictive setting considered in the present work where interactions involve only the center of mass.

First, we discuss in details the classical case of the overdamped Langevin process. Then, we briefly mention similar results for the kinetic Langevin process, Hamiltonian Monte Carlo and unadjusted numerical schemes.

\subsection{Overdamped Langevin dynamics}\label{sec:ovLangevin}

The free energy (at temperature $1$) associated to the energy~\eqref{eq:E} is \begin{equation}
\label{eq:freeEnergy}
\mathcal F(\rho) =  \mathcal E(\rho) + \int_{\R^d} \mu \ln\mu \,. 
\end{equation}
Its Wasserstein gradient flow is the PDE
\begin{equation}
\label{eq:meanfieldEDP}
\partial_t \rho_t = \na\cdot \po \co \na V_0 + \na h_0(m_\mu) \cf  \rho_t \pf + \Delta \rho_t \,.
\end{equation}
This is the mean-field limit of the  particle system  $(\bX_t)_{t\geqslant 0}$ on $\R^{dN}$ solving 
\begin{equation}
\label{eq:overdampedUN}
\dd \bX_t = -\na U_N(\bX_t)\dd t + \sqrt{2}\dd \mathbf{B}_t 
\end{equation}
with $\mathbf{B}$ a $dN$-dimensional Brownian motion. We denote by $(P_t^N)_{t\geqslant 0}$ the associated Markov semi-group, so that  the law of $\bX_t$ is $\nu_0 P_t^N$ when $\bX_0  \sim \nu_0$. Its invariant measure is $\rho_\infty^N \propto e^{-U_N}$.

\begin{assu}
\label{assu:2}
 Assumption~\ref{assu:LSIN1} holds, and so does the uniform LSI~\eqref{unifLSI} for some $\lambda>0$. Moreover, $V$ is one-sided Lipschitz continuous, i.e.~\eqref{eq:one-sided}.
\end{assu}

For our purpose -- that is, ultimately, the study of the initial process with $h=0$ -- a convenient statement is the following (used in the proof of Theorem~\ref{thm:main1}).

\begin{thm}\label{thm:CVHmain}
Under Assumption~\ref{assu:2}, the PDE \eqref{eq:meanfieldEDP} admits a unique stationary solution $\rho_*$ and there exists $C>0$ such that for all $N\geqslant 1$, $t\geqslant 1$ and $\nu\in\mathcal P_2(\R^{dN})$, 
\begin{equation}
\label{eq:CVparticuleNL*}
 \mathcal H(\nu P_t^N|\rho_*^{\otimes N})  \leqslant C e^{-\lambda t}\mathcal W_2^2(\nu,\rho_*^{\otimes N}) +C \,.
\end{equation}
\end{thm}

This is proven along the rest of this section, which presents more consequences of interest of the uniform LSI.

\begin{thm}\label{thm:meanfieldconsequence}
Under Assumption~\ref{assu:2}:
\begin{enumerate}
\item The free energy admits a unique global minimizer $ \rho_*$, which is the unique stationary solution of~\eqref{eq:meanfieldEDP} and satisfies
\[ \rho_* = \Gamma(\rho_*)\,,\]
where
\[\Gamma(\rho) \propto \exp \po - V_0(x) - x\cdot \na h_0(m_\rho) \pf\,.\] 
\item The following global non-linear LSI holds: writing $\overline{\mathcal F}(\rho) = \mathcal F(\rho) -  \mathcal F(  \rho_*)$,
\begin{equation}
\label{eq:GNLLSI}
\forall \rho\in\mathcal P_2(\R^d)\,,\qquad  \overline{\mathcal F}(\rho)   \leqslant \frac{1}{\lambda }\mathcal I \po \rho|\Gamma(\rho)\pf \,,
\end{equation}
and for all $t\geqslant 0$, along~\eqref{eq:meanfieldEDP},
\begin{equation}\label{decayFt}
\overline{\mathcal F}(\rho_t)   \leqslant e^{-\lambda t}  \overline{\mathcal F}(\rho_0)\,.
\end{equation}
Moreover, the following global non-linear Talagrand inequality holds:
\begin{equation}
\label{eq:nonlinTalagrand}
\forall \rho\in\mathcal P_2(\R^d)\,,\qquad  \mathcal W_2^2(\rho,  \rho_*) \leqslant \frac{4}{\lambda } \overline{\mathcal F}(\rho)   \,.
\end{equation}
\item Global propagation of chaos occurs at stationarity:  
\begin{equation}
\label{eq:chaos-stationarity}
\sup_{N\geqslant 1}  \mathcal W_2^2(  \rho_*^{\otimes N}, \rho_\infty^N)+ \mathcal H(  \rho_*^{\otimes N}| \rho_\infty^N) + \mathcal I(\rho_*^{\otimes N}|\rho_\infty^N) +  \mathcal H( \rho_\infty^N| \rho_*^{\otimes N} ) + \mathcal I(\rho_\infty^N| \rho_*^{\otimes N}) <\infty  \,.
\end{equation}
\end{enumerate}

\end{thm}
\begin{proof}
A lower bound and the existence of a global minimizer $\tilde\rho_*$ for $ \mathcal F(\rho)$ are obtained as in the proof of Lemma~\ref{lem:Ftheta}.

Reasoning as in the proof of Lemma~\ref{lem:PoCstationnaire} we get that
\begin{equation}
\label{eq:tilderhoHC}
\mathcal W_2^2(  \rho_*^{\otimes N}, \rho_\infty^N) \leqslant \frac{4}{\lambda} \mathcal H(  \rho_*^{\otimes N}| \rho_\infty^N) \leqslant \frac{4}{\lambda^2} \mathcal I( \rho_*^{\otimes N}|\rho_\infty^N) \leqslant C
\end{equation}
for $C>0$ independent from $N$. Similarly,
\begin{align*}
\mathcal I \po \rho_\infty^N |\rho_*^{\otimes N} \pf &= \int_{\R^{dN}} \left| \na \ln \frac{\rho_\infty^N }{\rho_*^{\otimes N}  }\right|^2 \rho_\infty^N   \\
&=  N \int_{\R^{dN}} | \na h_0(y_*) - \na h_0(\bar x) |^2 \rho_\infty^N (\bx)\dd \bx \\
& \leqslant  N \|\na^2 h_0\|_\infty^2 \int_{\R^{dN}} | y_\theta - \bar x |^2 \rho_\infty^N (\bx)\dd\bx  \\
&\leqslant 2 \|\na^2 h_0\|_\infty^2\co \mathrm{Var} \po \rho_{*}\pf + \mathcal W_2^2\po \rho_\infty^N ,\rho_*^{\otimes N}\pf\cf \,.
\end{align*}
Using~\eqref{eq:tilderhoHC} and the LSI for $\rho_*^{\otimes N}$ concludes the proof of~\eqref{eq:chaos-stationarity}.

 Introduce the $N$-particle free energy
\[\mathcal F^N(\mu) = N \int_{\R^{dN}} \mathcal E(\pi(\bx)) \mu(\dd\bx) + \int_{\R^{dN}} \mu \ln\mu\]
for $\mu\in\mathcal P_2(\R^{dN})$, which is such that
\[\mathcal H(\mu|\tilde \rho_\infty^N) = \mathcal F^N(\mu) - \mathcal F^N(\rho_\infty^N)\,.\]
For $\rho\in\mathcal P_2(\R^d)$,
\begin{multline*}
\frac1N \mathcal F^N(\rho^{\otimes N}) = \int_{\R^{d}} V_0 \rho - \frac{\kappa(N-1)}{2N} \int_{\R^{2d}} xy \rho(x)\rho(y)\dd x \dd y \\
+ \frac1N \int_{\R^d} x^2 \rho(x)\dd x +  \int_{\R^{dN}} h(\bar x) \rho^{\otimes N}(\dd\bx) + \int_{\R^{d}} \rho \ln\rho
\underset{N\rightarrow \infty}\longrightarrow\  \mathcal F(\rho)\,.
\end{multline*}
Then,
\[\frac1N \mathcal F^N(\rho_\infty^N) = \frac1N \mathcal F^N( \rho_*^{\otimes N}) - \frac1N  \mathcal H( \rho_*^{\otimes N}|\rho_\infty^N)  \underset{N\rightarrow \infty}\longrightarrow\ \mathcal F(\rho_*)\,.\] 
We have thus obtained that 
\[\frac1N  \mathcal H(\rho^{\otimes N}|\rho_\infty^N)  \underset{N\rightarrow \infty}\longrightarrow\ \mathcal F(\rho)-\mathcal F(\rho_*)\,. \]
On the other hand, for $\rho\in\mathcal P_2(\R^d)$,
\begin{align*}
\frac1N \mathcal I(\rho^{\otimes N}|\rho_\infty^N) &=  \int_{\R^{dN}} |\na_{x_1} \ln\rho(x_1) + \na_{x_1} U_N(\bx)|^2 \rho^{\otimes N}(\bx)\dd \bx \\
 &=  \int_{\R^{dN}} |\na_{x_1} \ln\rho(x_1) + \na V_0(x_1) - \na h_0(\bx)|^2 \rho^{\otimes N}(\bx)\dd \bx \\
 &\underset{N\rightarrow\infty}\longrightarrow  \int_{\R^{d}} |\na_{x_1} \ln\rho(x_1) + \na V_0(x_1) - \na h_0(m_{\rho})|^2 \rho(x_1)\dd x_1 
\end{align*}
using that $\|\na^2 h_0\|_\infty<\infty$. 

Letting $N\rightarrow \infty$ in the LSI of $\rho_\infty^N$ gives~\eqref{eq:GNLLSI}, which implies the exponential decay~\eqref{decayFt} and the Transport inequality~\eqref{eq:nonlinTalagrand}, see \cite[Lemma 4]{MonmarcheReygner}. It also implies that any stationary solution of~\eqref{eq:meanfieldEDP}, which is thus a fixed point of $\Gamma$, is a global minimizer of $\mathcal F$. The bound~\eqref{eq:tilderhoHC} holds for any such global minimizer $ \rho_*$. Denoting by $\rho_\infty^{1,N}$ the first $d$-dimensional marginal of $\rho_\infty^N$, \eqref{eq:tilderhoHC} and the scaling properties of the $\mathcal W_2$ for indistinguishable particles gives
\[\mathcal W_2^2( \rho_*,\rho_\infty^{1,N}) \leqslant \frac1N \mathcal W_2^2( \rho_*^{\otimes N},\rho_\infty^N)  \underset{N\rightarrow\infty}\longrightarrow 0\,. \] 
Hence the minimizer is unique, characterized as the limit of $\rho_\infty^{1,N}$.
\end{proof}

By the usual sub-additivity of the relative entropy with respect to tensorized measures (see e.g. \cite[Lemma 5.1]{Chen1}), the global estimate~\eqref{eq:chaos-stationarity} implies that the existence of some constant $C>0$ such that for all $N\geqslant 1$ and all $k\in\cco 1,N\ccf$,
\[\|\rho_\infty^{kN} - \rho_*^{\otimes k}\|_{TV}^2 + \mathcal W_2^2\po \rho_\infty^{kN} , \rho_*^{\otimes k}\pf + \mathcal H \po \rho_\infty^{kN} | \rho_*^{\otimes k}\pf \leqslant \frac{Ck}{N}\,.\]
In fact the recent works \cite{lacker2022quantitative,LackerLeFlem,MonmarcheRenWang,RenSongboSize} have shown that it is possible to get estimates of order $k^2/N^2$ under additional assumptions. The results of these references do not apply immediately in our context as they are written for pairwise interactions. It is probably possible to extend their proofs (in particular of \cite{RenSongboSize}) to our case. However getting sharp estimates is not our focus and thus we  postpone this refinement to future works.

\bigskip

The next statement follows from \cite[Corollary 7, Proposition 23]{MonmarcheReygner} and \cite[Corollary 1.2]{RocknerWang}.

\begin{prop}\label{prop:regularize}
Under Assumption~\ref{assu:LSIN1}, assume moreover~\eqref{eq:one-sided}. Then:
\begin{enumerate}
\item For all $N\geqslant 1$, $\nu_0 \in\mathcal P_2(\R^{dN})$ and $t>0$,
\begin{equation}
\label{eq:W2HregularizeNparticules}
\mathcal  H\po \nu_0 P_t^N | \rho_\infty^N \pf \leqslant \frac{C}{1\wedge t} \mathcal W_2^2 \po \nu_0,\rho_\infty^N\pf\,.\end{equation}
\item For all $N\geqslant 1$, $t\geqslant 0$ and $\nu_0 \in\mathcal P_2(\R^{dN})$,
\begin{equation}
\label{eq:W2HregularizeNparticules+NL}
\mathcal  H\po \nu_0 P_{t+1}^N | \rho_\infty^{\otimes N} \pf \leqslant C \mathcal W_2^2 \po \nu_0P_t^N,\rho_\infty^N\pf + C\,.\end{equation}
\item Assuming furthermore \eqref{eq:GNLLSI} for some $\lambda>0$, along~\eqref{eq:meanfieldEDP}, for all $t>0$,
\begin{equation}
\label{eq:W2HregularizeNL}
\overline{\mathcal F}(\rho_t) \leqslant \frac{C}{1\wedge t} \mathcal W_2^2(\rho_0,\rho_*)\,.
\end{equation}
\end{enumerate}

\end{prop}

\begin{cor}\label{cor:consequences}
Under Assumption~\ref{assu:2}, there exists $C>0$ such that:
\begin{enumerate}
\item  for all $N\geqslant 1$, $t\geqslant 0$ and $\bx,\by\in\R^{dN}$,
\begin{equation}
\label{eq:boundTV}
\|\delta_{\bx} P_t^N - \delta_{\by} P_t^N \|_\infty^2 \leqslant C e^{-\lambda t} \po N+ |\bx|^2 + |\by|^2\pf\,.
\end{equation}
\item For all $\rho_0 \in \mathcal P_2(\R^d)$ and $t\geqslant 0$, along~\eqref{eq:meanfieldEDP},
\begin{equation}
\label{eq:W2contract}
\mathcal W_2^2(\rho_t,\rho_*) \leqslant C e^{-\lambda t}\mathcal W_2^2(\rho_0,\rho_*)\,.
\end{equation}
\item For all $N\geqslant 1$, $t\geqslant 0$ and $\nu\in\mathcal P_2(\R^{dN})$, 
\begin{equation}
\label{eq:W2contract-particules}
\mathcal W_2^2(\nu P_t^N,\rho_\infty^N) \leqslant C e^{-\lambda t}\mathcal W_2^2(\nu,\rho_\infty^N)\,.
\end{equation}
\end{enumerate}

\end{cor}
\begin{proof}
By taking $C \geqslant 2e^{-\lambda}$, the bound~\eqref{eq:boundTV} is trivial for $t\leqslant 1$ since the total variation norm is less than $2$. For $t\geqslant 1$, using the Wasserstein-to-entropy regularization~\eqref{eq:W2HregularizeNparticules} with the entropy decay implied by the LSI and the Pinsker inequality, we get
\begin{eqnarray*}
\|\delta_{\bx} P_t^N - \delta_{\by} P_t^N \|_\infty^2 &\leqslant & 2\|\delta_{\bx} P_t^N - \rho_\infty^N  2\|_\infty^2 +2 \|\rho_\infty^N  - \delta_{\by} P_t^N \|_\infty^2 \\
& \leqslant & 2  e^{-\lambda (t-1)}\co \mathcal H(\delta_{\bx} P_1^N |\rho_\infty^N) +  \mathcal H(\delta_{\by} P_1^N |\rho_\infty^N)\cf \\
& \leqslant & 2 C e^{-\lambda (t-1)}\co \mathcal W_2^2(\delta_{\bx} ,\rho_\infty^N) +  \mathcal W_2^2(\delta_{\by},\rho_\infty^N)\cf 
\\
& \leqslant & 4 C e^{-\lambda (t-1)}\co |\bx|^2 + |\by|^2 + 2 \int_{\R^{dN}} |\mathbf{z}|^2 \rho_\infty^N(\dd \mathbf{z}) \cf \,.
\end{eqnarray*}
The second moment of $\rho_\infty^N$ is of order $N$ thanks to~\eqref{eq:chaos-stationarity}, which concludes the proof of~\eqref{eq:boundTV}.

The bound~\eqref{eq:W2contract} (resp. \eqref{eq:W2contract-particules}) is immediately obtained by combining~\eqref{eq:W2HregularizeNL}, \eqref{decayFt} and \eqref{eq:nonlinTalagrand} (resp.~\eqref{eq:W2HregularizeNparticules}, the LSI and Talagrand inequality for $\rho_\infty^N$), at least for $t\geqslant 1$. For $t\leqslant 1$ we use that the $\mathcal W_2$ distance grows at most exponentially with time along $P_t^N$ or \eqref{eq:meanfieldEDP}, see the proof of \cite[Theorem 8]{MonmarcheReygner} for details.

\end{proof}

\begin{proof}[Proof of Theorem~\ref{thm:CVHmain}]
The existence and uniqueness of $\rho_*$ has already been proven with Theorem~\ref{thm:meanfieldconsequence}. Using Corollary~\ref{cor:consequences}, we bound
\begin{eqnarray*}
\mathcal W_2\po \nu P_t^N,\rho_*^{\otimes N}\pf  &\leqslant& \mathcal W_2\po \nu P_t^N,\rho_\infty^{N}\pf + \mathcal W_2\po \rho_\infty^{N},\rho_*^{\otimes N}\pf \\
&\leqslant & C e^{-\lambda t/2}\mathcal W_2(\nu,\rho_\infty^N) + \mathcal W_2\po \rho_\infty^{N},\rho_*^{\otimes N}\pf \\
&\leqslant & C e^{-\lambda t/2}\mathcal W_2(\nu,\rho_*^{\otimes N}) +(1+C) \mathcal W_2\po \rho_\infty^{N},\rho_*^{\otimes N}\pf
\end{eqnarray*}
and conclude with~\eqref{eq:chaos-stationarity} and \eqref{eq:W2HregularizeNparticules+NL}.
\end{proof}

\subsection{Other processes}\label{sec:otherproc}

The LSI is particularly associated to the overdamped Langevin dynamics~\eqref{eq:overdampedUN} since it is equivalent to the exponential decay of the relative entropy with respect to $\rho_\infty^N$ along this flow, and the mean-field limit~\eqref{eq:meanfieldEDP} is the Wasserstein gradient flow of the associated free energy~\ref{eq:freeEnergy}. However, this inequality can also be involved in the study of other processes, as we now discuss.

Consider for instance the (kinetic) Langevin process with potential $U_N$, which is the process $(\bX_t,\mathbf{V}_t)_{t\geqslant 0}$ on $\R^{2dN}$ solving 
\[\left\{\begin{array}{rcl}
\dd \bX_t & = & \mathbf{V}_t\dd t\\
\dd \mathbf{V}_t &= & -\na U_N(\bX_t)\dd t - \gamma \mathbf{V}_t \dd t + \sqrt{2\gamma} \dd \mathbf{B}_t\,,
\end{array}\right.\]
for some friction parameter $\gamma>0$. Denote $(Q_t^N)_{t\geqslant 0}$ the associated semi-group and consider on $\R^{2d}$ the probability measure $\nu_*$ given by
\[\nu_*  = \rho_* \otimes \mathcal N(0,I_{dN}) \,,\]
with $\rho_*$ as in Theorem~\ref{thm:meanfieldconsequence}.

The equivalent of Theorem~\ref{thm:CVHmain}  in the kinetic case holds:

\begin{thm}\label{thm:kinetic}
Under Assumption~\ref{assu:2}, assume moreover that $\|\na^2 V\|_\infty <\infty$. Then, there exists $C>0$ such that for all $N\geqslant 1$, $t\geqslant 1$ and $\nu\in\mathcal P_2(\R^{2dN})$, 
\begin{equation}
\label{eq:CVparticuleNL*-kin}
 \mathcal H(\nu Q_t^N|\nu_*^{\otimes N})  \leqslant C e^{-\lambda t}\mathcal W_2^2(\nu,\nu_*^{\otimes N}) +C \,.
\end{equation}
\end{thm}

\begin{proof}
The structure of the proof is exactly the same as Theorem~\ref{thm:CVHmain}. All the equivalent in the kinetic case (under the additional condition that $\|\na^2 V\|_\infty <\infty$, see \cite{multipliers} to go beyond this restriction) of the results stated in Theorem~\ref{thm:meanfieldconsequence}, Proposition~\ref{prop:regularize} or Corollary~\ref{cor:consequences} can be found in \cite{GuillinMonmarche,Chen_etal,MonmarcheReygner}. In particular, the kinetic version of  \eqref{eq:W2HregularizeNparticules+NL} is \cite[Lemma 5.3]{Chen_etal}. Besides, the propagation of chaos at stationarity~\eqref{eq:chaos-stationarity} immediately gives the same result in  the kinetic case, since the invariant measure of $(Q_t^N)_{t\geqslant 0}$ is $\rho_\infty^N \otimes \mathcal N(0,I_{dN})$, and thus it has the same marginal in velocity as $\nu_*^{\otimes N}$.
\end{proof}

Similarly, the LSI allows to establish the long-time convergence of idealized Hamiltonian Monte Carlo~\cite{idealized} or of Generalized Langevin diffusion processes~\cite{ottobre2011asymptotic,pavliotis2021scaling} and of unadjusted numerical schemes, such as the Euler-Maruyama scheme for the overamped process~\eqref{eq:overdampedUN} (often called ULA for Unadjusted Langevin Algorithm) \cite{Suzukietal,VempalaWibisono} and Euler or splitting schemes for Hamiltonian Monte Carlo and the kinetic Langevin process \cite{Camrudetal,Katharina,fu2023mean}.

 To get exactly a result of the form~\eqref{eq:CVparticuleNL*} requires a $\mathcal W_2$/entropy approximate regularization in the spirit of~\eqref{eq:W2HregularizeNparticules+NL}. The usual change of measure method to prove this has been used to control the difference between a continuous-time diffusion process and its numerical scheme in e.g. \cite{dalalyan2012sparse,chewi2021optimal}. By  combining these with the mean-field approximation~\eqref{eq:W2HregularizeNparticules+NL} (i.e. following the proof but coupling a time-discretized chain of $N$ interacting particles with $N$ independent continuous-time non-linear McKean-Vlasov processes, or possibly $N$ independent non-linear discrete-time chains having the correct stationary distribution as in \cite{idealized}) we expect to get a bound of the form
\[\mathcal  H\po \nu_0 Q^{\lfloor 1/h\rfloor} | \rho_\infty^{\otimes N} \pf \leqslant C \mathcal W_2^2 \po \nu_0,\rho_\infty^N\pf + C + C N h^p\,,\]
with  $Q$ the transition of the unadjusted scheme, $h$ its step-size, $p=2$ for first-order schemes and $p=4$ for second order schemes, and $C$ independent from $N$ and $h$. Combining it with the long-time convergence of the references of the previous paragraph yields a result similar to~\eqref{eq:CVparticuleNL*} and \eqref{eq:CVparticuleNL*-kin}.

\bigskip

From such a bound, the last missing ingredient to transfer this information to the initial process (with $h=0$) and  conclude with a result similar to Theorem~\ref{thm:main1}  is a bound on the exit time of the domain $\{\bx\in\R^{dN},\ h(\bar x)=0\}$. In this paper we only present this last step for the overdamped Langevin diffusion and postpone  to future works the extension   of the full analysis to other processes and schemes.

\section{Exponentiality of exit times}\label{sec:exponential}

This section uses the settings and notations of Sections~\ref{sec:UnifLSI} and \ref{sec:consequences}. Its purpose is the study of the exit time $\tau_N = \inf\{t\geqslant 0,\ \bar X_t \notin \mathcal D\}$ where $\bX_t$ solves \eqref{eq:overdampedUN}, $\bar X_t = \frac1N \sum_{i=1}^N X_t^i$ and $\mathcal D \subsetneq \R^d$ is some metastable domain. Under Assumptions~\ref{assu:2}, we consider $\rho_*$ given by Theorem~\ref{thm:meanfieldconsequence}, $m_*=m_{\rho_*}$ and set $t_N=\mathbb E _{\rho_*^{\otimes N}} (\tau_N)$. The main result of this section is the following.

\begin{thm}\label{thm:exitdanslecasLSI}
Under Assumptions~\ref{assu:2}, let $\mathcal D \subset \R^d$ be such that $\mathcal D^c$ has a non-empty interior and there exists $r_1>0$ with $\mathcal B(m_*,r_1) \subset \mathcal D$. Let $r_2>0$ be such that a solution of \eqref{eq:meanfieldEDP} initialized in $\mathcal B_{\mathcal W_2}(\rho_*,2r_2)$ remains in $\mathcal B_{\mathcal W_2}(\rho_*,r_1/2)$  for all times (this exists thanks to \eqref{eq:W2contract}). For all $k>0$, there exists $C>0$ such that for all $N\geqslant 1$ and all $\bx\in\R^{dN}$ with $\mathcal W_2(\pi(\bx),\rho_*)\leqslant r_2$,
\begin{eqnarray}\label{eq:thm-exit-1}
\sup_{s\geqslant 0} |\mathbb P_{\bx}\po \tau_{\mathcal D^c} > s t_N \pf - e^{-s}|  & \leqslant & \frac{C}{N^k}
\end{eqnarray}
and 
\begin{eqnarray}\label{eq:thm-exit-2}
\left| \frac{\mathbb E_{\bx}\po \tau_{N} \pf}{t_N}-1\right| & \leqslant &  \frac{C}{N^k} \,.
\end{eqnarray}
\end{thm}

This is established by invoking \cite[Theorem 3.4]{Eva}. The fact that the assumptions of  this general result are satisfied in our case and give Theorem~\ref{thm:exitdanslecasLSI} is summarized here:

\begin{prop}
\label{prop:checkCond}
In the settings of Theorem~\ref{thm:exitdanslecasLSI}, there exists $T>0$ such that for all $k>0$ there exists $C>0$  such that, for all $N\geqslant 1$, writing $\mathcal K_N = \{\bx\in\R^{dN},\ \mathcal W_2 (\pi(\bx),\rho_*)\leqslant r_2\}$, $\mathcal D_N=\{\bx\in\R^{dN},\ \bar x\in\mathcal D\}$ and $\tau_{\mathcal K_N} = \inf\{t\geqslant 0,\ \bX_t \in \mathcal K_N\}$, 
\begin{eqnarray}
 \sup_{\bx\in\mathcal K_N} \ \mathbb P_{\bx} \po \tau_{N} \leqslant N^k \pf & \leqslant & \frac{C}{N^k}  \label{eqdef:generalexit3}  \\
 \sup_{\bx\in\mathcal D_N} \ \mathbb P_{\bx} \po \tau_N \wedge   \tau_{\mathcal  K_N} > T \pf & \leqslant &  \frac{C}{N^k} \label{eqdef:generalexit11} \\
\sup_{\bx,\by\in\mathcal K_N} \|\delta_{\bx} P_{N}^N - \delta_{\by} P_{N}^N\|_ {TV} & \leqslant  & \frac{C}{N^k} \label{eqdef:generalexit4}\\
\mathbb P_{\rho_*^{\otimes N}} \po \mathcal K_N^c \pf & \leqslant & \frac{C}{N^k} \label{eqdef:generalexitend} \,.
\end{eqnarray}
\end{prop}

The interpretation of the different conditions is the following. We see $\mathcal K_N$ as a stable center of $\mathcal D_N$. This stability is enforced by \eqref{eqdef:generalexit3}  which states that leaving from $\mathcal D_N$ in a  time less than polynomial in $N$ starting from $\mathcal K_N$ is unlikely. The center is also attractive in the sense that, starting from anywhere in $\mathcal D_N$, a process which has not exited $\mathcal D_N$ before time $T$ (independent from $N$) has likely ``fell down" to $\mathcal K_N$ (this is \eqref {eqdef:generalexit11}). The last ingredient to apply \cite[Theorem 3.4]{Eva} (in combination with \cite[Proposition 5.8]{Eva}) is a mixing/loss of memory within $\mathcal K_N$, described by \eqref{eqdef:generalexit4}.  The additional property~\eqref{eqdef:generalexitend}  ensures that, when applying \cite[Theorem 3.4]{Eva}, we can use $\rho_{*}^{\otimes N}$ as a reference measure, so that~\eqref{eq:thm-exit-1} and \eqref{eq:thm-exit-2} hold with $t_N$ defined in \eqref{eq:deftN}.

\begin{proof}
Let us check each condition one after the other.
\begin{itemize}
\item \textbf{Loss of memory in the center.} Applying~\eqref{eq:boundTV}, for any $\bx,\by\in\mathcal K_N$ and $t\geqslant 0$,
\begin{eqnarray*}
\|\delta_{\bx} P_{t}^N - \delta_{\by} P_{t}^N \|_\infty^2 &\leqslant &  C e^{-\lambda t} \po N+ |\bx|^2 + |\by|^2\pf\\
& = & C N e^{-\lambda t} \po 1+  \mathcal W_2^2(\pi(\bx),\delta_0) + \mathcal W_2^2(\pi(\by),\delta_0)  \pf\\
& \leqslant  & 4  C N e^{-\lambda t} \po 1+  r_2^2 + \mathcal W_2^2(\rho_\infty,\delta_0)   \pf\,.
\end{eqnarray*}
Applying this with $t=N$ gives~\eqref{eqdef:generalexit4} for any $k>0$ for some $C>0$ (depending on $k$).

\item \textbf{Fast return to the center.} Fix $k>0$. Thanks to Lemma~\ref{lem:CVfast}, there exists $T_1>0$ such that $\mathcal W_2(\rho_{T_1}^{\bx},\rho_*) \leqslant r_2/2$ for all $\bx\in\R^{dN}$, where $\rho^{\bx}$ is the solution of the non-linear equation~\eqref{eq:meanfieldEDP} with initial condition $\rho_0^{\bx}= \pi(\bx)$.  After a time $1$, according to Lemma~\ref{lem:moments_instantanés}, for any $\ell \geqslant 2$, there exists $R>0$ such that
\[\sup_{\bx\in\R^{dN}} \mathbb P_{\bx} \po \frac{1}{N}\sum_{i=1}^N |X_1^i|^\ell \geqslant R \pf = \underset{N\rightarrow \infty}{\mathcal O}(N^{-k})\,.\]
Combining this with Proposition~\ref{prop:PoCrhoX}, we get that 
\begin{equation}
\label{loc:combine}
\sup_{\bx\in\R^{dN}} \mathbb P_{\bx}\po \sup_{t\in[0,T_1]} \mathcal W_2\po \pi(\bX_{1+t}) , \rho_t^{\bX_{1}}\pf \geqslant \frac{r_2}{2} \pf  =\underset{N\rightarrow \infty}{\mathcal O}(N^{-k})\,.
\end{equation}
As a consequence, setting $T=1+T_1$,
\begin{multline}
\sup_{\bx\in\R^{dN}} \mathbb P \po \mathcal W_2\po \pi(\bX_{T}),\rho_*\pf \geqslant     r_2 \pf \\
\leqslant \sup_{\bx\in\R^{dN}} \mathbb P_{\bx}\po  \mathcal W_2\po \pi(\bX_{1+T_1}) , \rho_{T_1}^{\bX_{1}}\pf \geqslant \frac{r_2}{2} \pf =\underset{N\rightarrow \infty}{\mathcal O}(N^{-k})\,,\label{loc:pw2}
\end{multline}
which concludes the proof of~\eqref{eqdef:generalexit11}.

\item \textbf{Stability in the center.} Take $T$ as before. For $j \in\N$, consider the event
\[A_j = \left\{\sup_{t\in[jT,(j+1)T]} \mathcal W_2\po \pi(\bX_t),\rho_*\pf \leqslant r_1,\ \mathcal W_2\po \pi(\bX_{(j+1)T}),\rho_*\pf \leqslant r_2 \right\}\]
Our goal is to prove that 
\[\sup_{j\in\N} \mathbb P(A_{j+1}^c|A_j) =\underset{N\rightarrow \infty}{\mathcal O}(N^{-k})\,.\]
By the Markov property, it is sufficient to prove that, for all $k>0$,
\begin{equation}
\label{loc:Ak}
\sup_{\bx \in \mathcal K_N} \mathbb P_{\bx}(A_1^c) =\underset{N\rightarrow \infty}{\mathcal O}(N^{-k})\,.
\end{equation}
Once this is obtained, using that $\mathcal B(m_*,r_1)\subset \mathcal D$ so that the event $\{\tau_N \leqslant t\}$ implies that $\mathcal W_2(\pi(\bX_s),\rho_*) > r_1$ for some $s\in[0,t]$, we get~\eqref{eqdef:generalexit3} (for all $k>0$) by bounding
\begin{align*}
\sup_{\bx\in\mathcal K_N} \mathbb P \po \tau_N \leqslant N^{k/2}\pf & \leqslant \mathbb P \po \sup_{t\in[0,N^{k/2}]} \mathcal W_2\po \pi(\bX_t),\rho_*\pf  > r_1 \pf \\
&\leqslant \mathbb P(A_1^c) + \sum_{2\leqslant j\leqslant N^{k/2}} \mathbb P(A_j^c|A_{j-1}) \\
& \leqslant N^{k/2} \underset{N\rightarrow \infty}{\mathcal O}(N^{-k}) = \underset{N\rightarrow \infty}{\mathcal O}(N^{-k/2})\,.
\end{align*}
It thus remains to prove~\eqref{loc:Ak}. As we already have~\eqref{loc:pw2}, we only need
\begin{equation}
\label{loc:Ak2}
\sup_{\bx \in \mathcal K_N} \mathbb P_{\bx}\po \sup_{t\in[0,T]} \mathcal W_2\po \pi(\bX_t),\rho_*\pf > r_1\pf  =\underset{N\rightarrow \infty}{\mathcal O}(N^{-k})\,.
\end{equation}
To prove this, we will show that, before some small (fixed deterministic) time $t_0>0$, $\pi(\bX_t)$ has probably not exited from $\mathcal B_{\mathcal W_2}(\rho_*,2r_2)$, and after that it probably stays close to $\rho_t^{\bx}$ which, by definition of $r_1$ and $r_2$, remains in $\mathcal B_{\mathcal W_2}(\rho_*,r_1/2)$ for all times.

More specifically, let $\bx\in\R^{dN}$ and $\tilde \bY_0 \sim \rho_*^{\otimes N}$. Let $\sigma$ be a permutation of $\cco 1,N\ccf$ such that
\[\mathcal W_2^2\po \pi(\bx),\pi(\tilde \bY_0)\pf = \frac1N \sum_{i=1}^N |x_i - \tilde Y_0^{\sigma(i)}|^2\,.\]
Then $\bY_0 := (\tilde Y_0^{\sigma(1)},\dots,\tilde Y_0^{\sigma(N)}) \sim \rho_*^{\otimes N}$.  Consider $\bX$ and $\bY$ two solutions of \eqref{eq:overdampedUN} (driven by the same Brownian motion) with respective initial conditions $\bx\in\R^{dN}$ and $\bY_0$. Using the one-sided Lipschitz condition~\eqref{eq:one-sided} and that $\|\na^2 h\|_\infty<\infty$, there exists $L'>0$ such that, almost surely, for all $t\geqslant 0$,
\[|\bX_t - \bY_t| \leqslant e^{L't} |\bX_0-\bY_0|\,.\]
We bound 
\begin{align*}
\mathcal W_2\po \pi(\bX_t),\rho_*\pf &\leqslant \mathcal W_2\po \pi(\bX_t),\pi(\bY_t)\pf + \mathcal W_2\po \pi(\bY_t),\rho_*\pf \\
&\leqslant \frac{e^{L't}}{\sqrt{N}} |\bX_0-\bY_0| + \mathcal W_2\po \pi(\bY_t),\rho_*\pf \\
&=  e^{L't} \mathcal W_2\po \pi(\bX_0),\pi(\bY_0)\pf + \mathcal W_2\po \pi(\bY_t),\rho_*\pf \\
&\leqslant   e^{L't} \co r_2+  \mathcal W_2\po \rho_*,\pi(\bY_0)\pf \cf  + \mathcal W_2\po \pi(\bY_t),\rho_*\pf\,.
\end{align*}
Taking $t_0>0$ such that $e^{L't_0} < 3/2$, 
\[\mathbb P\po \sup_{t\in[0,t_0]}\mathcal W_2\po \pi(\bX_t),\rho_*\pf  > 2r_2 \pf \leqslant \mathbb P\po \sup_{t\in[0,t_0]}\mathcal W_2\po \pi(\bY_t),\rho_*\pf  > \frac{r_2}5 \pf =\underset{N\rightarrow \infty}{\mathcal O}(N^{-k}) \]
for all $k>0$ thanks to Proposition~\ref{prop:finiteTimePOC}.

As a conclusion,
\begin{multline*}
\sup_{\bx \in \mathcal K_N} \mathbb P_{\bx}\po \sup_{t\in[0,T]} \mathcal W_2\po \pi(\bX_t),\rho_*\pf > r_1\pf 
\leqslant \sup_{\bx \in \mathcal K_N} \mathbb P_{\bx}\po \sup_{t\in[0,t_0]}\mathcal W_2\po \pi(\bX_t),\rho_*\pf  > 2r_2 \pf 
\\
+ \sup_{\bx\in\R^{dN}} \mathbb P_{\bx}\po \sup_{t\in[0,T-t_0]} \mathcal W_2\po \pi(\bX_{t_0+t}) , \rho_t^{\bX_{t_0}}\pf > \frac{r_1}{2} \pf  = 
\underset{N\rightarrow \infty}{\mathcal O}(N^{-k}) 
\end{multline*}
for all $k>0$, reasoning as for~\eqref{loc:combine}.
\item \textbf{Reference measure.}  Finally, \eqref{eqdef:generalexitend} is an immediate consequence of \cite[Theorem 2]{FournierGuillin}, using that $\rho_*$ has exponential moments.
\end{itemize}

\end{proof}

\begin{cor}\label{cor:expoexittimemodif}
In the settings of Theorem~\ref{thm:exitdanslecasLSI}, there exist $a,\eta>0$ such that, for all $N \geqslant 1$, $t_N \geqslant \eta e^{a N}$. 
\end{cor}

\begin{proof}
Since \cite[Theorem 3.4]{Eva} applies, we get that
\[ \left|\frac{\mathbb E_{\nu_1}(\tau_N)}{\mathbb E_{\nu_2}(\tau_N)} -1 \right|  \underset{N\rightarrow \infty}{\longrightarrow} 0  \]
for all $\nu_1,\nu_2\in\mathcal P(\R^{dN})$ such that, for $i=1,2$,
\[  \mathbb P_{\nu_i}\po \mathcal K_N^c\pf  \underset{N\rightarrow \infty}{\longrightarrow} 0  \,. \]
We have already seen with \eqref{eqdef:generalexitend} that this condition is satisfied for $\nu_2= \rho_*^{\otimes N}$. Let us show that it also holds for $\nu_1  =\rho_\infty^N$. Let $\bX,\bY$ be an optimal $\mathcal W_2$ coupling of $\rho_\infty^N$ and $\rho_*^{\otimes N}$. Then
\begin{multline*}
\mathbb P_{\rho_\infty^N}(\mathcal K_N^c)  \leqslant \mathbb P \po \mathcal W_2\po \pi(\bX),\pi(\bY) \pf \geqslant r_2/2\pf + \mathbb P \po \mathcal W_2\po \pi(\bX),\rho_* \pf \geqslant r_2/2\pf \\
\leqslant \frac{4\mathbb E \po |\bX-\bY|^2\pf }{N r_2^2} +  \mathbb P \po \mathcal W_2\po \pi(\bX),\rho_* \pf > r_2/2\pf\underset{N\rightarrow \infty}{\longrightarrow} 0\,, 
\end{multline*} 
thanks to \eqref{eq:chaos-stationarity}, using that $\mathbb E \po |\bX-\bY|^2\pf = \mathcal W_2^2 (\rho_\infty^N,\rho_*^{\otimes N})$.

We have thus obtained that $t_N \simeq \mathbb E_{\rho_\infty^N}(\tau_N)$ as $N\rightarrow \infty$. Fix any $T_0>0$, and consider the events $A_k = \{\exists t\in[kT_0,(k+1)T_0],\ \bX_t \notin \mathcal D\}$ for $k\in\N$. Since $\rho_\infty^N$ is invariant for the dynamics, by the Markov property, $\mathbb P_{\rho_\infty^N}(A_k) = \mathbb P_{\rho_\infty^N}(A_0)$ for all $k\in\N$. According to \cite[Theorem 5.1]{dawsont1987large} or, more simply in our situation, \cite[Theorem 2.9]{bolley2007quantitative}, there exists $a,C>0$ such that $  \mathbb P_{\rho_\infty^N}(A_0) \leqslant C e^{-aN}$. Hence, for any $k\in\N$,
\[\mathbb P_{\rho_\infty^N} \po \tau_N < k T_0\pf \leqslant \sum_{j=0}^{k-1} \mathbb P_{\rho_\infty^N}(A_0)  \leqslant C k e^{-aN}\,.\]
Applying this with $k = k_0: =\lfloor e^{aN}/(2C)\rfloor$ shows that $\mathbb E_{\rho_\infty^N} \po \tau_N \pf  \geqslant k_0 T_0/2 $, which concludes.
\end{proof}

\begin{rem}\label{rem:DawsonGartner}
By using \cite[Theorem 5.1]{dawsont1987large} with \cite[Theorem 3]{dawson1986large} in the proof of Corollary~\ref{cor:expoexittimemodif}, we get a more precise result. Indeed, it gives
\[\frac1N \limsup_{\N\rightarrow\infty} \mathbb P_{\rho_\infty^N}(A_0) \leqslant - \inf\{   \mathcal F(\mu) - \mathcal F(\rho_*) :\ \mu\in\mathcal P_2(\R^d),m_{\mu}\notin \mathcal D\} \,.  \]
The proof of Corollary~\ref{cor:expoexittimemodif} then shows that
\[\frac1N \liminf_{\N\rightarrow\infty} t_N \geqslant\inf\{   \mathcal F(\mu) - \mathcal F(\rho_*) :\ \mu\in\mathcal P_2(\R^d),m_{\mu}\notin \mathcal D\} \,. \]
In dimension $1$ in the double well case, by taking $\mathcal D=[\varepsilon,\infty)$ with an arbitrarily small $\varepsilon>0$, combined with Theorem~\ref{thm:mainExit}, this proves essentially  the lower bound in \cite[Theorem 4]{dawson1986large}.
\end{rem}

\section{The modified energy}\label{sec:modified}

In this section, we consider   $m_*\in\R^d$, $V,W\in\mathcal C^2(\R^d,\R)$, $\mathcal D\subsetneq \R^d$, $\sigma^2>0$ as in Assumption~\ref{assu:main-result}. Our goal is to design a convex $\mathcal C^2$ Lipschitz-continuous function $h:\R^d\rightarrow \R$  with bounded second-order derivative such that the modified Gibbs measure $\tilde \rho_\infty^N$ given in~\eqref{eq:modigiedGibbs} satisfies a uniform LSI~\eqref{unifLSI}. Thanks to Theorem~\ref{thm:LSIN}, this is done by checking Assumption~\ref{assu:LSIN2}. More precisely, we set
\begin{equation}
\label{loc:V0}
V_0(x) = \frac{V(x) + |x|^2/2}{\sigma^2}\,,\qquad \kappa = \frac{1}{\sigma^2 }\,,\qquad h_0(m) =   h(m) - \frac{\kappa}{2}|m|^2\,.
\end{equation}
 With these notations, the Gibbs measure considered in Theorem~\ref{thm:LSIN} is indeed  $\tilde \rho_\infty^N$, and the function $f$ from Assumption~\ref{assu:main-result} can be written
\begin{equation}
f(m)  = \frac{\int_{\R^d} x \exp \po - V_0(x) + \kappa x\cdot m \pf\dd x  }{\int_{\R^d} \exp \po - V_0(x) + \kappa x\cdot m \pf  \dd x }\,.
\end{equation}

As in Assumption~\ref{assu:main-result}, we consider separately the one-dimensional and general cases.

\subsection{One-dimensional case}\label{sec:modified-dim1}

The goal of this section is to prove the following:

\begin{prop}\label{prop:hmodifdim1}
Under Assumption~\ref{assu:main-result} with $d=1$, given $V_0$ and $\kappa$ in~\eqref{loc:V0} there exists  a  convex  function $h\in\mathcal C^2(\R,\R)$ with $\|\na h\|_\infty,\|\na^2 h\|_\infty<\infty$, such that $h(m)=0$ for all $m\in\mathcal D$ and Assumption~\ref{assu:LSIN2} is satisfied for $h_0(m) =  h(m) - \frac{\kappa}{2}|m|^2$.
\end{prop}

This is an immediate corollary of Proposition~\ref{prop:dim1w} stated and proven below. We start with the useful properties of the function $f$ for our study.

\begin{prop}\label{prop:f}
Under Assumption~\ref{assu:main-result} with $d=1$, $f$ is an increasing Lipshitz continuous one-to-one map from $\R$ to $\R$, with  $f(m)/m$ vanishing as $|m|\rightarrow +\infty$. For $m\in (-\infty,m_*)\cap \mathcal D$,  $f(m)<m$ and  for $m\in (m_*,\infty)\cap \mathcal D$, $f(m)>m$.
\end{prop}

\begin{proof}
Since $f'(m) = \kappa \mathrm{Var}(\nu_m)>0$ for all $m\in\R$, we get that $f$ is increasing.  The fact that $f(m)<m$ for $m\in (-\infty,m_*)\cap \mathcal D$ and $f(m)>m$ for $m\in (m_*,\infty)\cap \mathcal D$ is then a consequence of the fact $m_*$ is the unique fixed point of $f$ in $\mathcal D$, with $0< f'(m_*) < 1$. Using that $x\mapsto V_0(x) - \kappa x\cdot m$ is the sum of a strongly convex function (with a lower bound on the curvature independent from $m$) and a bounded function (independent from $m$), as in Lemma~\ref{lem:LSImicro} we get that $\nu_m$ satisfies a Poincaré inequality independent from $m$. In particular, $\mathrm{Var}(\nu_m)$ is bounded independently from $m$, which shows that $f$ is Lipschitz continuous. Besides, $\nu_m$ is invariant for the Markov generator $L_m$ on $\R$ given by
\[L_m \varphi(x) = -(V_0'(x) - \kappa m) \varphi'(x) + \varphi''(x)\,.\]
Taking $\varphi(x)=x^2$ and using that, under Assumption~\ref{assu:main-result}, $x V_0'(x) \geqslant c|x|^\beta - C $ for all $x\in\R$ for some $c,C>0$ (with $\beta>2$),
\begin{align*}
0 = \int_\R L_m \varphi(x) \nu_m(\dd x) & \leqslant 2\int_\R \co - c|x|^\beta + C + 1 +  \kappa |m| |x| \cf \nu_m\\
& \leqslant \int_\R \co - c  |x|^\beta + 2(C+1)+  C' |m|^{\frac{\beta}{\beta-1}}   \cf \nu_m
\end{align*}
for some $C'>0$ independent from $m$. Using Jensen inequality, this shows that
\[\int_\R |x|^2 \nu_m(\dd x) = \underset{|m|\rightarrow \infty }{\mathcal O} \po |m|^{\frac{2}{\beta-1}}\pf = \underset{|m|\rightarrow \infty }o \po |m|^{2}\pf\,.\]
By Jensen inequality again, this shows that $f(m)=o(|m|)$. 

By comparing two diffusion processes associated respectively to $L_m$ and $L_{m'}$ with $m>m'$ with the same initial condition and driven by the same Brownian motion, we see that $\nu_m$ is stochastically larger than $\nu_{m'}$, so that $\int_{\R}\varphi \dd \nu_m \geqslant \int_{\R} \varphi \dd \nu_{m'}$ for all non-decreasing function $\varphi$. For $m>0$, we bound
\[f(x) = \int_{\R} x \nu_{m}(\dd x)   \geqslant \int_{-\infty}^0 x \nu_{0}(\dd x) + \int_{0}^\infty x \nu_{m}(\dd x) \geqslant \int_{-\infty}^0 x \nu_{0}(\dd x) + \frac12 \mathrm{med}(\nu_m) \,,  \]
with $\mathrm{med}(\nu_m)$ the median of $\nu_m$. It only remains to show that the latter goes to infinity as $m\rightarrow \infty$ (the case $m\rightarrow -\infty$ being similar). Fix $x_0 \in \R$ and consider $\varphi(x)= (x_0-x)^2_+$. For this $x_0$, we can find $m$ large enough so that $L_m \varphi(x) \leqslant 3 -\sqrt{m}\1_{x\leqslant   x_0-1}$ and, then
\[0 =  \int_\R L_m \varphi(x) \nu_m(\dd x) \leqslant - \sqrt{m} \nu_m\po (-\infty,x_0-1]\pf + 3\] 
In particular, for $m$ large enough, $\mathrm{med}(\nu_m)\geqslant x_0-1$, which concludes since $x_0$ is arbitrary.
\end{proof}

Let us now define explicitly a suitable $h$. Under Assumption~\ref{assu:main-result} with $d=1$, the domain $\mathcal D$ is of the form $[a,\infty)$, $[a,b]$ or $(-\infty,b]$ for $a,b\in\R$. We only detail the case where $\mathcal D= [a,\infty)$, the others being similar. This means that $f$ admits a unique fixed point $m_*>a$ in $\mathcal D$. Moreover, since $f(m)/|m| \rightarrow 0$ as $m\rightarrow -\infty$ while $f(m) >m $ for $m\in[a,m_*)$, there exists a fixed point $m_- < a$ of $f$ such that $f$ has no fixed point on $(-\infty,m_-)$. By continuity there also exists $\varepsilon>0$ such that $f$ has no fixed point in $[a-\varepsilon,a]$.

 In the following we write $f^{-1}$ the inverse of $f$. We set $a'= \kappa f^{-1}(a)$,  $a''= \kappa f^{-1}(a-\varepsilon)$ and notice that, necessarily, $\kappa m_- = \kappa f^{-1}(m_-) < a''  < a' < \kappa f^{-1}(m_*)=\kappa m_*$.  Consider an increasing function $r\in\mathcal C^2(\R,\R)$ with the following properties:
\begin{align}
r(z) &= z \qquad \forall z\geqslant a'\label{eq:r1}\\
r(\kappa m_-) & \geqslant a'' \label{eq:r2}\\
r(z) &= r(\kappa m_-)  + z- \kappa m_- \qquad \forall z\leqslant \kappa m_- \label{eq:r3}\\
0< r'(z) & \leqslant 1 \qquad \forall y\in\R\,.\label{eq:r4}
\end{align} 
We define $h:\R\rightarrow \R$ by $h(m_*)=0$ and
\begin{equation}
 \label{eq:h'inverser}
\forall y\in\R\,,\qquad  h'(y) = r^{-1}\po \kappa f^{-1}(y)\pf -  \kappa f^{-1}(y) \,.
 \end{equation}
See Figure~\ref{fig:f} for a sketch. This specific choice has the following nice consequences: 

\begin{prop}\label{prop:dim1w}
The function $h$ defined in~\eqref{eq:h'inverser} is convex, Lipschitz-continuous with bounded second derivative. It satisfies $h(y)=0$ for all $y\geqslant a$.

Moreover, with this $h$, the function $w$ defined in~\eqref{eq:defw} is such that
\begin{equation}
\label{eq:w'}
w'(\theta) = \frac{\theta}{\kappa} - f\po \frac{r(\theta)}{\kappa}\pf\,.  
\end{equation}
Its unique critical point is $\kappa m_*$ and there exists $\eta>0$ such that for all $\theta\in\R$,
\begin{equation}
\label{eq:wcoercif}
\po \theta - \kappa m_* \pf w'(\theta) \geqslant \eta \po \theta - \kappa m_*\pf ^2\,.
\end{equation}
Finally, $w''$ is lower bounded.
\end{prop} 

\begin{rem}\label{rem:convex}
Differentiating~\eqref{eq:w'} for $\theta \geqslant a'$ (for which $r(\theta)=\theta$) gives $w''(\theta) = (1 - f'(\theta/\kappa))/\kappa$. In the double-well case~\eqref{eq:doublewell} below the critical temperature, we can apply Proposition~\ref{prop:dim1w} with $m_*=\mu_+$ and $\mathcal D=[a,\infty)$ for any $a>0$. Nevertheless, in this case, $0$ is a fixed point of $f$ with  $f'(0)>1$, which means that $f'(\theta/\kappa)>1$ for $\theta>0$ small enough. Hence, if we want to keep $h=0$ up to an arbitrarily small $a>0$ (to study the initial process~\eqref{eq:particules} up to the vicinity of the saddle point), we cannot ensure $w$ to be strongly convex. This is why we use Proposition~\ref{prop:LSIcontractif} in the proof of Lemma~\ref{lem:LSImacro} instead of the Bakry-Emery criterion.
\end{rem}
\begin{proof}
The fact that $h'(y)=0$ (and thus $h(y)=0$) for all $y\geqslant a = f( a'/\kappa)$ directly follows from~\eqref{eq:r1}.  Moreover, for $y\leqslant m_-$, $\kappa f^{-1}(y) \leqslant \kappa m_-$, hence by~\eqref{eq:r3} 
\[h'(y) =  \kappa m_- - r(\kappa m_-)\,.  \] 
As a continuous function which is constant over $(-\infty,m_-]$ and $[a,\infty)$, $h'$ is bounded.

Differentiating~\eqref{eq:h'inverser}  gives
\[ h''(y) = \kappa (f^{-1})'(y) \co (r^{-1})'\po \kappa f^{-1}(y)\pf - 1\cf \,. \]
Since $(r^{-1})'$ is $1$ outside a compact set, $h''$ is bounded, and moreover $(r^{-1})'\geqslant 1$ thanks to~\eqref{eq:r4}, so that $h$ is convex (since $(f^{-1})'\geqslant 0$).

Using the notations of Lemma~\ref{lem:Ftheta}, for $\theta\in\R$,
\[y_\theta = f_{\theta}(y_\theta) = f\po \frac{\theta - h'(y_\theta)}\kappa \pf\,,\]
and moreover $y_\theta$ is the unique solution to this equation. The design of $h$ is meant so that 
\[y_\theta = f\po \frac{r(\theta)}\kappa\pf \,,\]
which, thanks to the uniqueness of the fixed point, is proven by checking that, indeed,
\[f\po \frac{\theta - h'\po f(r(\theta)/\kappa )\pf }\kappa \pf = f\po \frac{r(\theta)}\kappa\pf\,.  \] 

Recall the definition~\eqref{eq:g}, so that~\eqref{eq:reformule_w} reads
\[w(\theta) = \frac{\theta^2}{2\kappa} + g_\theta(y_\theta)\,. \]
Differentiating~\eqref{eq:g} in $y$ gives
\[g_\theta'(y) = h''(y)\po y - f_\theta(y)\pf\,,\]
and thus $g_\theta'(y_\theta) = 0$. Moreover, $\partial_\theta g_\theta(y) =-f_\theta(y)$. Hence
\[w'(\theta) = \frac{\theta}{\kappa} - f_{\theta}(y_\theta) =  \frac{\theta}{\kappa} - y_\theta\,.  \]
Using the expression of $y_\theta$ gives~\eqref{eq:w'}. This equation and the facts that $r( \kappa m_*)= \kappa m_*$ (since $m_* \kappa > a'$) and $m_*=f(m_*)$ show that $\kappa m_*$ is a critical point of $w$.

Let us show that $w$ has no other critical point. First, over $[a',\infty)$, $r(\theta)=\theta$, so a critical point of $w$ in this region would be $\kappa$ times a fixed point of $f$, but we know that the only fixed point in $(a,\infty)$ is $m_*$. Second, on $[a-\varepsilon,m_*)$, $f(x) > x$, which means that for $\theta\in [a'',a']$, $\theta/\kappa < f(\theta/\kappa) \leqslant f\po r(\theta)/\kappa\pf$ (the fact $r(\theta) \geqslant \theta$ for all $\theta\in\R$ follows from~\eqref{eq:r1} and \eqref{eq:r4}).  The same argument shows that $w$ has no critical point in $(-\infty,\kappa m_-)$ because $f(x)>x$ on $(-\infty,m_-)$.   Finally, for $\theta\in[\kappa m_-,a'')$, $r(\theta)\geqslant r(\kappa m_-) \geqslant a''$ and thus $f(r(\theta)/\kappa) \geqslant  f(a''/\kappa)> a''/\kappa> \theta/\kappa$. We have thus established that $\kappa m_*$ is the unique critical point of $w$.

In a small neighborhood around $\kappa m_*$, $r(\theta) = \theta$. Differentiating~\eqref{eq:w'} at $\kappa m_*$ thus gives
\[w''(\kappa m_*)  = \frac{1 - f'(m_*)}{\kappa}>0\,, \] 
thanks to Assumption~\ref{assu:main-result}. Hence, \eqref{eq:wcoercif} holds in a neighborhood of $\kappa m_*$ for some $\eta>0$. As a consequence it holds over any compact set (for some $\eta$ depending a priori on the compact set)  since, away from $\kappa m_*$, $\theta \mapsto \po \theta - \kappa m_* \pf w'(\theta)/\po \theta - \kappa m_*\pf ^2$ is a positive continuous function (since $w$ has no other critical point). Using that $r$ grows at most linearly at infinity and that, according to Proposition~\ref{prop:f}, $f(\theta) = o(|\theta|)$, we get from~\eqref{eq:w'} that  $ \po \theta - \kappa m_+ \pf w'(\theta)/\po \theta - \kappa m_+\pf ^2$ converges to $1/\kappa$ as $|\theta|\rightarrow \infty$. This concludes the proof that \eqref{eq:wcoercif} holds on $\R$ for some $\eta>0$.

Finally, the lower bound on $w''$ is a direct consequence of~\eqref{eq:w'} and Proposition~\ref{prop:f} (since $f$ and $r$ are both $\mathcal C^1$ with bounded derivative).

\end{proof}

\subsection{Multi-dimensional case}

The goal of this section is to prove the following:

\begin{prop}\label{prop:hmodifdimD}
Under Assumption~\ref{assu:main-result} with $\mathcal D = \mathcal B(m_*,r)$ for some $r>0$, given $V_0$ and $\kappa$ in~\eqref{loc:V0} there exists  a  convex  function $h\in\mathcal C^2(\R^d,\R)$ with $\|\na h\|_\infty,\|\na^2 h\|_\infty<\infty$, such that $h(m)=0$ for all $m\in\mathcal D$ and Assumption~\ref{assu:LSIN2} is satisfied for $h_0(m) =  h(m) - \frac{\kappa}{2}|m|^2$.
\end{prop}

The useful properties of $f$ are the following.

\begin{prop}\label{prop:fdimd}
Under Assumption~\ref{assu:main-result}, there exist $\delta,R>0$ such that, on the one hand $|\na f(m)|\leqslant 1-\delta$ and $|f(m)-m_*|\leqslant r$ for all $m\in \mathcal B(m_*,r+\delta)$ and, on the other hand, $|f(m)-m_*|\leqslant |m-m_*|/\max(4\kappa,4\kappa^2)$ for all $m\in\R^d$ such that $f(m)\notin \mathcal B(m_*,R)$. Moreover, $\na f(m)$ is a symmetric positive definite matrix for all $m\in\R^d$ and $\|\na f\|_\infty<\infty$.
\end{prop}
\begin{proof}
Since $|\na f|<1$ over $\mathcal B(m_*,r)$, by continuity and compactness we can find $\delta'>0$ such that $|\na f| \leqslant 1-\delta'$ over $\mathcal B(m_*,r+\delta')$. For $m\in\mathcal B(m_*,r+\delta)$ with $\delta =  \delta' \min(1,r)$,
\[|f(m)-m_*| = |f(m)-f(m_*)| \leqslant (1-\delta')|m-m_*| \leqslant (1-\delta')(r+\delta) \leqslant r\,.\]
 The existence of $R$ follows from the fact $f(m)=o(|m|)$ as $|m|\rightarrow \infty$, which is proven as in Proposition~\ref{prop:f}. Next, we compute
\[\na f(m)  =  \kappa \po \int_{\R^d} x x^T \nu_m - f(m) f(m)^T \pf = \kappa \mathrm{Covar}(\nu_m) \,.\] 
This shows that $\na f$ is symmetric positive definite. The fact it is uniformly bounded then follows from the Poincaré inequality satisfied by $\nu_m$, with a constant independent from $m$, as in Proposition~\ref{prop:f}.
\end{proof}

\begin{proof}[Proof of Proposition~\ref{prop:hmodifdimD}]
Consider $\delta,R$ as in Proposition~\ref{prop:fdimd}. Set $h(m)= H\po |m-m_*|\pf$ where $H\in\mathcal C^2(\R_+,\R_+)$ is non-decreasing, convex and such that
\begin{align*}
H(s) = 0 &\qquad  \forall s \in[0,r]\\
H'(s) =  2\kappa  s &\qquad \forall  s\in[r+\delta/2,R]\\
H'(s) \leqslant 2 \kappa  s & \qquad \forall s\geqslant 0\\
H''(s)=0 & \qquad \forall s\geqslant R+1
\end{align*}
In particular, denoting $e=(m-m_*)/|m-m_*|$,
\begin{equation}
\label{loc:nah}
\na h (m) = e H'(|m-m_*|)\,,\qquad \na^2 h(m) =\frac{\Id - e e^T}{|m-m_*|} H'(|m-m_*|)\ + e e^T H''(|m-m_*|)\,.
\end{equation}
This $h$ clearly satisfies all conditions of Proposition~\ref{prop:hmodifdimD} apart from Assumption~\ref{assu:LSIN2}. 
Using the notations of Lemma~\ref{lem:Ftheta},
\begin{equation}
\label{loc:naw}
\na w(\theta) = \frac{\theta }{\kappa}  - y_\theta 
\end{equation}
with $y_\theta$    the unique solution of 
\begin{equation}
\label{loc:ytheta}
y_\theta = f_{\theta}(y_\theta) = f\po \frac{\theta - \na h(y_\theta)}\kappa \pf\,.
\end{equation}
Differentiating this equality with respect to $\theta$ gives, writing $B= \na f\po \frac{\theta - \na h(y_\theta)}\kappa \pf$,
\[\na_\theta y_\theta = \co \kappa \Id + \na^2 h(y_\theta)B   \cf^{-1} B = B^{-1}  \co \kappa B^{-1} + \na^2 h(y_\theta)   \cf^{-1} B  \,. \]
Notice that $\kappa B^{-1} + \na^2 h(y_\theta)  $ is indeed invertible, since it is symmetric positive definite. As a consequence
\begin{equation}
\label{loc:na2w}
\na^2 w(\theta)  =\frac{1}{\kappa} \Id -  B^{-1} \co \kappa B^{-1} +  \na^2 h(y_\theta)     \cf^{-1} B \,.
\end{equation}
This implies that $\na_\theta y_\theta = B^{-1} \co \kappa B^{-1} +  \na^2 h(y_\theta)     \cf^{-1} B$  is a symmetric matrix. Since it is conjugated with $[\kappa B^{-1} +  \na^2 h(y_\theta)]^{-1}$, which is also symmetric, we get
\[|\na_\theta y_\theta | = |[\kappa B^{-1} +  \na^2 h(y_\theta)]^{-1} | \leqslant \frac{|B|}{\kappa}\,, \]
using again that $\na^2 h(y_\theta)$ is positive. This shows that $\na^2 w$ is lower bounded, since $|B|\leqslant \|\na f\|_\infty$.

As in the proof of Proposition~\ref{prop:hmodifdim1} in dimension $1$, the fact that $h=0$ in the vicinity of $m_*$ and that $m_*$ is a fixed point of $f$ shows that $\theta_*:=\kappa m_*$ is a critical point of $w$. 

In the remaining of the proof, we will show that there exists $\eta>0$ such that,
\begin{equation}
\label{loc:either}
\forall \theta\in\R^d\,,\qquad \text{either}\qquad \na^2 w(\theta) \geqslant \eta \Id\,,\qquad \text{or}\qquad (\theta - \theta_*) \cdot \na w(\theta) \geqslant \eta   |\theta-\theta_*|^2\,.
\end{equation}
This will conclude the proof that Assumption~\ref{assu:LSIN2} holds, hence the proof of Proposition~\ref{prop:hmodifdimD}. Indeed, from~\eqref{loc:either}, for any $\theta\in\R^d$, we can distinguish two cases. If $\na^2 w(\theta_p) \geqslant \eta   \Id$ for all $p\in[0,1)$ with $\theta_p=(1-p)\theta+p\theta_*$, we simply write
\[ (\theta - \theta_*) \cdot \na w(\theta) = \int_0^1  (\theta - \theta_*) \cdot  \na^2 w(\theta_p)(\theta - \theta_*) \dd p \geqslant \eta   |\theta - \theta_*|^2\,.\]
Otherwise, the set $\{p\in[0,1), (\theta_p - \theta_*) \cdot \na w(\theta_p) \geqslant \eta  |\theta_p-\theta_*|^2\}$ is not empty so we can consider its infimum $q$. By continuity, $(\theta_{q} - \theta_*) \cdot \na w(\theta_q) \geqslant \eta   |\theta_q-\theta_*|^2$, and then
\begin{multline*}
 (\theta - \theta_*) \cdot \na w(\theta)  = (\theta - \theta_*) \cdot \na w(\theta_q) +  \int_0^1  (\theta - \theta_*) \cdot  \na^2 w(\theta_{uq})(\theta - \theta_q) \dd u \\
   \geqslant q\eta   |\theta - \theta_*|^2 + (1-q) \eta  |\theta - \theta_*|^2 = \eta  |\theta - \theta_*|^2\,.
\end{multline*}
We   prove \eqref{loc:either} by distinguishing cases according to the value of $|\theta-\theta_*|$. 
\begin{itemize}
\item[$\circ$] If $|\theta-\theta_*|\leqslant \kappa  (r+\delta) $, thanks to Proposition~\ref{prop:fdimd}, $f(\theta/\kappa) \in \mathcal B(m_*,r)$. Hence, $\na^2 h\po f(\theta/\kappa)\pf = 0$, which shows that $y_\theta = f(\theta/\kappa)$ (by the definition of $y_\theta$ as the unique solution of~\eqref{loc:ytheta}). As a consequence, over $\mathcal B(\theta_*,r+\delta)$,
\[\na^2 w(\theta) = \frac{1}{\kappa} \po  \Id - \na f\po \frac{\theta}{\kappa}\pf \pf  \geqslant \frac{\delta }{\kappa} \Id\,. \]
\item[$\circ$] If $|\theta-\theta_*|/\kappa \geqslant r+\delta $, we consider three sub-cases.
\begin{itemize}
\item  If $|y_\theta - m_*|\leqslant r+\delta/2$, then, taking the scalar product of~\eqref{loc:naw} with $(\theta-\theta_*)$ gives
\[ (\theta-\theta_*) \cdot \na w(\theta)\geqslant \frac{1}{\kappa}|\theta-\theta_*|^2 - |\theta-\theta_*| |y_\theta - m_*| \geqslant \frac{1}{\kappa} \po 1 - \frac{r+\delta/2}{r+\delta} \pf|\theta-\theta_*|^2\,. \]
\item If $|y_\theta - m_*|\in[ r+\delta/2,R]$, then $\na^2 h(y_\theta) =  2\kappa  \Id$ (recall~\eqref{loc:nah}). As a consequence,~\eqref{loc:na2w} reads
\[\na^2 w(\theta) = \frac{1}{\kappa} \Id - \po \kappa B^{-1} + 2 \kappa \Id \pf^{-1} \geqslant  \frac{1}{2\kappa}   \Id\,, \]
using that $B$ is positive.
\item If $|y_\theta - m_*| \geqslant R$, thanks to Proposition~\ref{prop:fdimd} and by the definition~\eqref{loc:ytheta} of $y_\theta$, 
\[
|y_\theta-m_*| \leqslant \frac{|\theta - \theta_* + \kappa \na h(y_\theta)|}{4\kappa\max(1,\kappa)}
\leqslant \frac1{4\kappa}|\theta - \theta_*| + \frac{1}{2}|y_\theta - m_*|\,, \]
where we used that $|\na h(m)|\leqslant 2\kappa|m-m_*|$ for all $m\in\R^d$. Taking as before the scalar product of~\eqref{loc:naw} with $(\theta-\theta_*)$ gives
\[ (\theta-\theta_*) \cdot \na w(\theta)\geqslant \frac{1}{\kappa}|\theta-\theta_*|^2 - |\theta-\theta_*| |y_\theta - m_*| \geqslant \frac{1}{2\kappa}|\theta-\theta_*|^2\,. \]
\end{itemize}

\end{itemize}
This concludes the proof of Proposition~\ref{prop:hmodifdimD}.
\end{proof}

\section{Illustrations}\label{sec:illustration}

All simulations below concern the one-dimensional  double-well potential~\eqref{eq:doublewell}.

 The means $m_+$ and $m_-$ of the stationary solutions $\mu_+$ and $\mu_-$, as a function of $\sigma$, are represented in Figure~\ref{fig:pointsfixes}, obtained by iterating the function $f$ from \eqref{eq:def-f} starting at $m=1$ or $m=-1$ until two consecutive iterations differ from less than $10^{-5}$ (so that we observe $m_+=m_-=0$ above $\sigma_c\simeq  0.68$).

\begin{figure}
\centering
\includegraphics[scale=0.4]{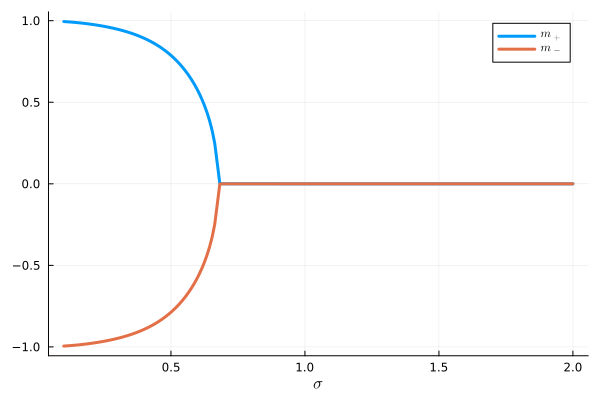}
\caption{Mean of the stationary solutions as a function of $\sigma$.}
\label{fig:pointsfixes}
\end{figure}

 Figures~\ref{fig:fast} and \ref{fig:fast-2} illustrate that, below the critical temperature, the initial convergence to the stationary solutions $\mu_-$ and $\mu_+$ is fast, namely independent from $N$. At temperature $\sigma= 0.5 < \sigma_c$, we see that at time $T=10$ these metastable states are already reached for both $N=10^4$ and $N=10^5$ particles. The absence of macroscopic motion afterwards up to time at least $T=10^3$ is illustrated by the very small fluctuations of the barycenter. The equivalent of  Figures~\ref{fig:fast} and \ref{fig:fast-2} at temperature larger than $\sigma_c$ is displayed in Figure~\ref{fig:highTemp}.
 
With respect to the set-up of  Figures~\ref{fig:fast} and \ref{fig:fast-2}, in order to observe transitions, we increase the temperature to $\sigma = 0.64$, decrease the number of particles to $N=1000$ and increase time to $T=10^4$. A realization exhibiting one transition during this time is displayed in Figure~\ref{fig:transition}. The trajectory of the barycenter shows that the transition is brutal, i.e. the duration between the first time where the blue line reaches $m_-$ and the last time it was at $m_+$ is very short with respect to the time needed to wait for such an event.

\begin{figure}
\begin{subfigure}{.48\textwidth}
  \centering
  \includegraphics[width=.95\linewidth]{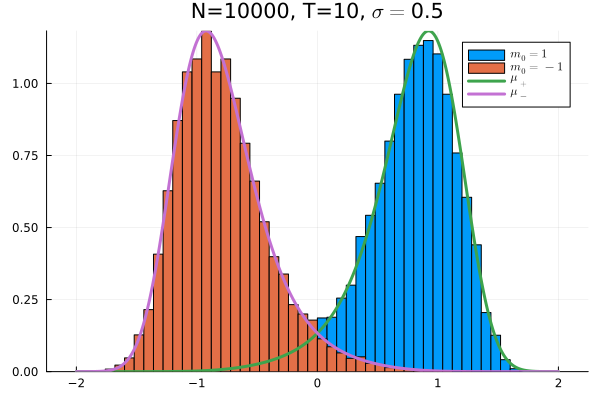}
\end{subfigure}%
\hspace{0.5cm}
\begin{subfigure}{.48\textwidth}
  \centering
  \includegraphics[width=.95\linewidth]{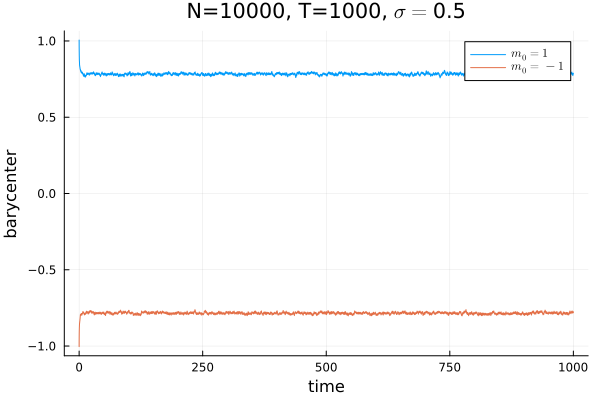}
\end{subfigure}
\caption{Fast convergence to stationary solutions, for $\sigma=0.5$ and $N=10^4$ particles initialized independently with $\mathcal N(m_0,1/4)$ with $m_0=1$ (in blue) or $m_0=-1$ (in orange). \emph{(Left)} At time $T=10$, histogram of the particles superimposed with the graph of $\mu_-$ and $\mu_+$. \emph{(Right)}. Trajectory $t\mapsto \bar X_t$ up to time $T=1000$.} 
\label{fig:fast}
\end{figure}

\begin{figure}
\begin{subfigure}{.48\textwidth}
  \centering
  \includegraphics[width=.95\linewidth]{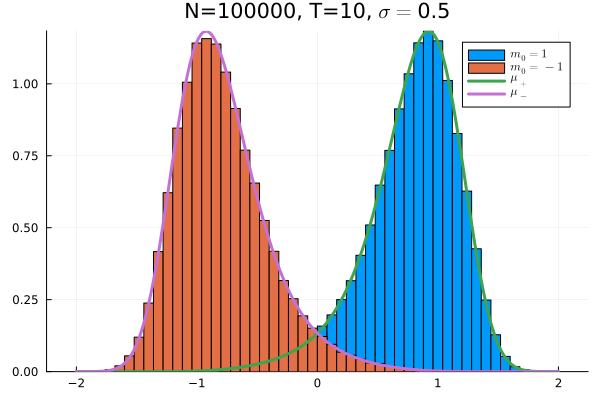}
\end{subfigure}%
\hspace{0.5cm}
\begin{subfigure}{.48\textwidth}
  \centering
  \includegraphics[width=.95\linewidth]{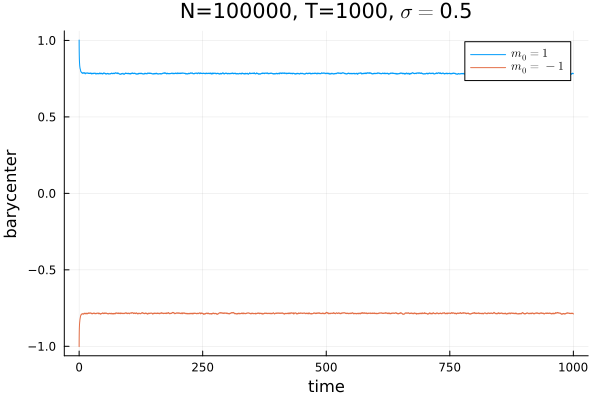}
\end{subfigure}
\caption{Same as Figure~\ref{fig:fast} except $N=10^5$.} 
\label{fig:fast-2}
\end{figure}

\begin{figure}
\begin{subfigure}{.48\textwidth}
  \centering
  \includegraphics[width=.95\linewidth]{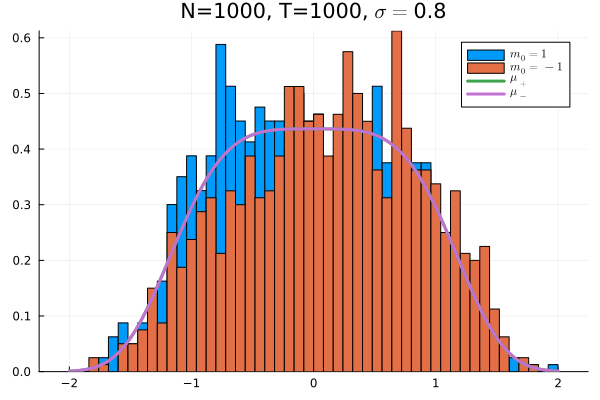}
\end{subfigure}%
\hspace{0.5cm}
\begin{subfigure}{.48\textwidth}
  \centering
  \includegraphics[width=.95\linewidth]{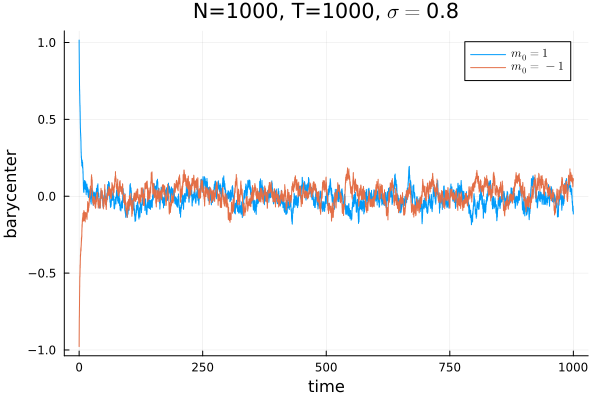}
\end{subfigure}
\caption{Same as Figure~\ref{fig:fast} except $N=10^3$, $T=10^3$ and $\sigma=0.8>\sigma_c$.} 
\label{fig:highTemp}
\end{figure}

\begin{figure}
\begin{subfigure}{\textwidth}
  \centering
  \includegraphics[width=.5\linewidth]{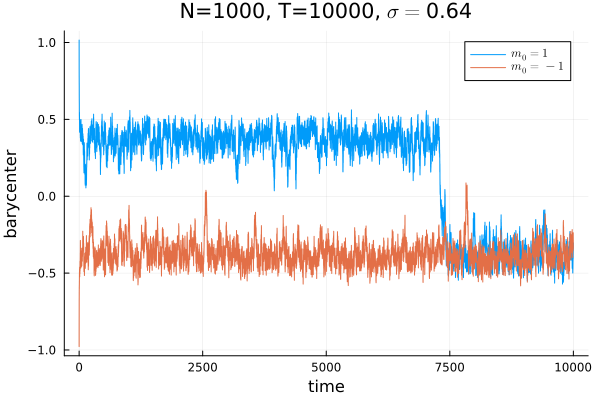}
\end{subfigure}

\begin{subfigure}{.48\textwidth}
  \centering
  \includegraphics[width=.95\linewidth]{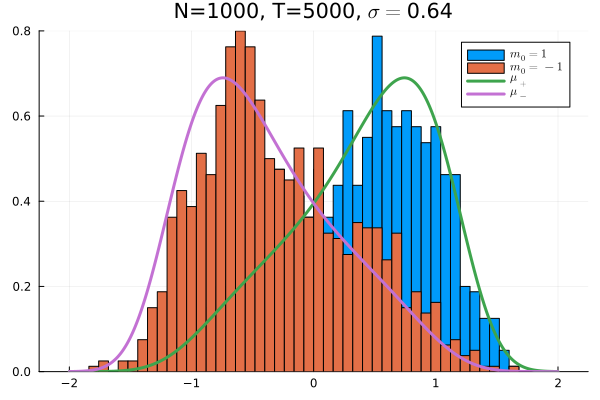}
\end{subfigure}%
\hspace{0.5cm}
\begin{subfigure}{.48\textwidth}
  \centering
  \includegraphics[width=.95\linewidth]{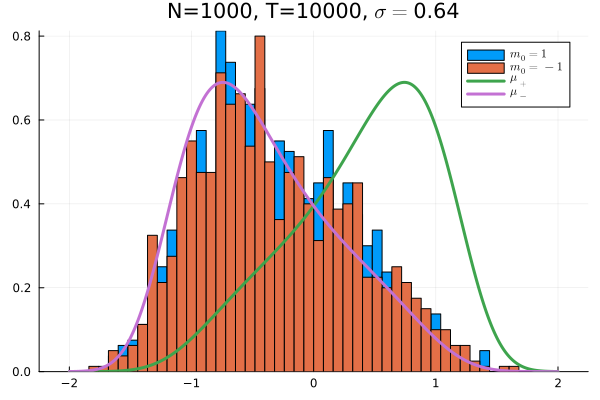}
\end{subfigure}
\caption{A metastable transition. Same initial conditions as in Figure~\ref{fig:fast}, but $\sigma=0.64$ and $N=10^3$. \emph{(Top)} Trajectory $t\mapsto \bar X_t$ up to time $T=10^4$. \emph{(Bottom)} Histogram of the particles superimposed with the graph of $\mu_-$ and $\mu_+$, \emph{(Left)} at time $T=5000$ before the transition and \emph{(Right)} at time $T=10^4$ after the transition.}  
\label{fig:transition}
\end{figure}

In Figure~\ref{fig:f} (Left) is shown the function $f$  at temperatures $\sigma\in\{1,0.68,0.1\}$. On the right, the previous curve for $\sigma=0.1$ is superimposed with the modified function
\[\tilde f(m) = f(r(m))= \int_{\R} x \tilde \nu_{m}(x)\dd x\,,\ \text{where}\  \tilde \nu_m(x) \propto \exp\po  - \frac{1}{\sigma^2}\co V(x) + \frac{1}{2}|x-m|^2 + x h'(m)\cf \pf\,,\]
still with $\sigma=0.1$, where $r$ and $h$ are in the spirit of the functions designed in Section~\ref{sec:modified-dim1}. As a consequence, the fixed-points of $\tilde f$ are in one-to-one correspondence with the stationary solutions of the modified process~\eqref{eq:particules-modif2} with drift~\eqref{loc:b}. We see that  $\tilde f$ is precisely designed to coincide with $f$ as long as possible while ensuring that it has a unique fixed point. This is an indication that the modified process is not metastable (as fully proven by applying Theorem~\ref{thm:LSIN}, with the explicit $h$ designed in Section~\ref{sec:modified}).

\begin{figure}
\begin{subfigure}{.48\textwidth}
  \centering
  \includegraphics[width=.9\linewidth]{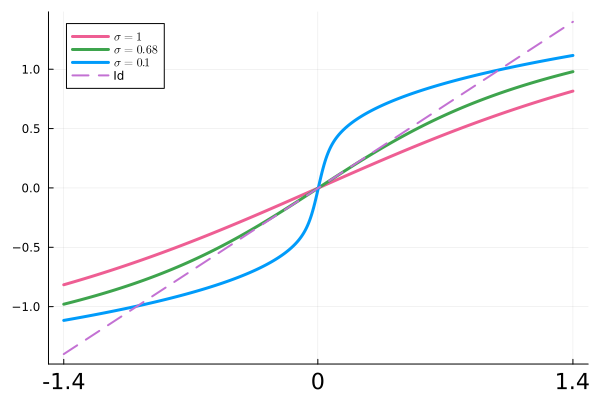}
\end{subfigure}%
\hspace{0.5cm}
\begin{subfigure}{.48\textwidth}
  \centering
  \includegraphics[width=.9\linewidth]{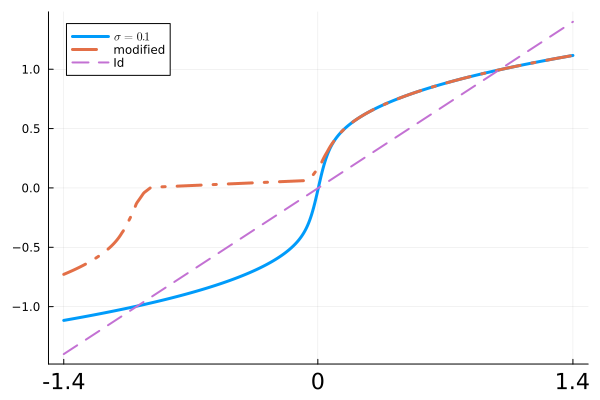}
\end{subfigure}
\caption{The initial fixed-point problem and its modification. \emph{(Left)} The function $f$ at three different temperatures: above, at and below criticality. \emph{(Right)} The function $f$ below the critical temperature and its modification with a non-zero $h$. } 
\label{fig:f}
\end{figure}

\section{Perspectives}

There is a number of directions in which our approach and results could be extended or improved. We have already mentioned the questions of obtaining sharp estimates   as in \cite{Lacker,RenSongboSize} or of considering kinetic processes or numerical scheme as in Section~\ref{sec:otherproc}. For this latter variation, it would be useful to remove the condition that $\na V$ grows faster than linearly (since we would rather assume that $ \|\na^2 V\|_\infty <\infty$), which should not be too difficult (this condition is convenient for the moment bounds estimates of Appendix~\ref{sec:moments} or to get that $\limsup_{|m|\rightarrow\infty}|f(m)|/|m|<1$ but it is not crucial). It would be also of interest to go beyond the case of quadratic interaction, and even of pair interactions. For this purpose, it would be relevant to simply Assumption~\ref{assu:LSIN2}.  Ideally, we could hope to prove that Assumption~\ref{assu:LSIN2} holds when a global non-linear LSI~\eqref{eq:GNLLSI} holds for the mean-field problem. In that case, Theorem~\ref{thm:LSIN}  would prove the conjecture from \cite{Pavliotis}, which is that a non-linear LSI implies a uniform-in-$N$ LSI for the particle system. Even in our context (interaction through the barycenter), this is  an open question.

In the present work we say nothing of saddle points of the free energy, for instance $\mu_0$ in the double-well case below the critical temperature. Studying these unstable stationary solutions will be useful to prove an upper bound on $t_N$ that matches the lower bound in Remark~\ref{rem:DawsonGartner} and thus conclude the proof of \cite[Theorem 4]{dawson1986large}. It will also be useful to characterize the basins of attraction of $\mu_-$, $\mu_0$ and $\mu_+$, which is an open question. The counter-example of \cite[Proposition 14]{MonmarcheReygner} shows that the sign of $m_{\rho_0}$ is not sufficient to characterize the long-time limit of $\rho_t$. A question is whether there exists $T>0$ such that the sign of $m_{\rho_T}$ determines this long-time limit. Due to the instantaneous bound on the second moment of $\rho_t$ and the $\mathcal W_2$ to $\mathcal{F}$ regularization, we know that, for any $t>0$, $\mathcal{F}(\rho_t)$ is bounded independently from $\rho_0\in\mathcal P_2(\R^d)$, and then it decays exponentially fast as long as $\rho_t$ avoids a neighborhood of $\mu_0$ (thanks to the local non-linear LSI proven in \cite{MonmarcheRenWang}, the decomposition similar to~\eqref{eq:decomposeFtheta} for $\mathcal F$ and the LSI for $\Gamma(\rho)$). Essentially, the information still missing to answer the previous question is whether, for any $T>0$, we can find $\rho_0$ as a small perturbation of $\mu_0$ such that the corresponding solution of~\eqref{eq:granularmedia} changes sign after time $T$.

\subsection*{Acknowledgments}

The research of P. Monmarché is supported by the project  CONVIVIALITY (ANR-23-CE40-0003) of the French National Research Agency  and by the European Research Council (ERC) under the European Union’s
Horizon 2020 research and innovation program, project EMC2 (grant agreement N°810367).

\bibliographystyle{plain} 
\bibliography{biblio}

\appendix

 \section{High-temperature LSI with Lyapunov conditions}

This section is devoted to the proof of the following:

 \begin{prop}\label{prop:LSIcontractif}
 Let $u \in \mathcal C^2(\R^d,\R)$ be such that there exists $c,k>0$ such that $x\cdot \na u(x) \geqslant c|x|^2$ and $\na^2 u(x) \geqslant -k $ for all $x\in\R^d$.  Then  there exists $\kappa>0$ such that for all $N\geqslant 1$,  the probability measure with density proportional to $e^{-Nu}$ satisfies a LSI with constant $\kappa/N$.
 \end{prop}

%

It is based on Lyapunov arguments for LSI. Let us recall \cite[Theorem 1.4]{BBCG} and   \cite[Theorem 1.2]{CattiauxGuillinWu} (or more precisely here \cite[Theorem 3.15]{MenzSchlichting} with $\Omega=\R^d$  since the constants are explicit).

\begin{thm}[Theorem 1.4 of \cite{BBCG}]\label{thm:BBCG}
Let $U\in\mathcal C^2(\R^d,\R)$. Assume that there exist $\theta,b,R>0$ and $W \in\mathcal C^2(\R^d,\R)$ with $W(x)\geqslant 1$ and
\[-\na U(x)\cdot \na W(x) + \Delta W(x) \leqslant -\theta W(x)+ b \1_{B(0,R)}(x) \]
for all $x\in\R^d$. Then $\mu \propto e^{-U}$ satisfies a Poincaré inequality with $(1+b\kappa_R)/\theta$, with $\kappa_R$ the Poincaré constant of the restriction of $\mu$ to $B(0,R)$.
\end{thm}

\begin{thm}[Theorem 3.15 of \cite{MenzSchlichting}]\label{thmCGW-MS}
Let $U\in\mathcal C^2(\R^d,\R)$. Assume that $\mu\propto e^{-U}$ satisfies a Poincaré inequality with constant $C_P$, and that there exist $K,\lambda,b> 0$  and a $\mathcal C^2$ function $W:\R^d \rightarrow [1,\infty)$ such that for all $x\in\R^d$, $\na^2 U(x) \geqslant - K$ and
\begin{equation}
\label{eq:LyapLSI}
-\na U(x)\cdot \na W(x) + \Delta W(x) \leqslant \po- \lambda |x|^2 + b \pf W(x)\,.
\end{equation}
Then,  writing $m_2=\int_{\R^d} |x|^2 \pi(\dd x)$, $\mu$ satisfies a LSI with constant
\[C_{LS} \leqslant 2\sqrt{\frac1\lambda \po \frac12 + C_P(b+\lambda m_2)  \pf } + \frac{K\po 1+ 2C_P(b+\lambda m_2)\pf +4\lambda C_P}{2\lambda } \,.\]  
\end{thm}

\begin{proof}[Proof of Proposition~\ref{prop:LSIcontractif}]
We apply successively Theorems~\ref{thm:BBCG} and~\ref{thmCGW-MS} with $U = Nu$. Notice that for a given value of $N$ it is easily checked that these results apply under the context of Proposition~\ref{prop:LSIcontractif}, so that $e^{-N u}$ satisfies a LSI. In other words, we only have to prove the result for $N$ large enough.

 First, the condition that $x\cdot \na u(x) \geqslant c| x|^2$ implies that $u$ goes to infinity at infinity, hence it admits a critical point, which is necessarily at $0$ and has a positive-definite Hessian matrix. Hence, we can find $r_1>0$ such that $u$ is strongly convex on $B(0,r_1)$. By the Bakry-Emery criterion, the Poincaré constant $\kappa_{r_1}$ of the restriction of $e^{-N u}$ on $B(0,r_1)$ is bounded by $\kappa'/N$ for some $\kappa'>0$ independent from $N$. Taking $W(x) = 1+ |x|^2/2$, we see that 
\[- N\na u(x)\cdot \na W(x) + \Delta W(x) \leqslant  - c N |x|^2  + d\,. \]
For $N \geqslant  2 d /(cr_1^2)$, this gives, for $x\notin B(0,r_1)$, 
\[- N\na u(x)\cdot \na W(x) + \Delta W(x) \leqslant  - \frac{c}2 N |x|^2  \leqslant - \frac{c r_1^2}{2+r_1^2}  N W(x) \,. \]
For $x\in B(0,r_1)$, 
\[- N\na u(x)\cdot \na W(x) + \Delta W(x) \leqslant   d \leqslant  - \frac{c r_1^2}{2+r_1^2}  N W(x)  +  \frac{c r_1^2}{2(2+r_1^2)}  N r_1^2 + d   \,. \]
Applying Theorem~\ref{thm:BBCG}, we get that $e^{-N u}$ satisfies a Poincaré inequality with constant $\kappa''/N$ for some $\kappa''>0$ independent from $N$.

Second, we set $W(x) = e^{a N|x|^2}$ for some $a>0$ to be chosen. Then
\[-N \na u (x)\cdot \na W(x) + \Delta W(x) \leqslant  \co - 2a N^2 c|x|^2  + (2adN+4a^2 N^2 |x|^2)\cf   W(x)\,.\]
Taking $a=c/4$ yields~\eqref{eq:LyapLSI} with $\lambda = a c N^2 $ and $b=2adN$. Applying Theorem~\ref{thmCGW-MS} (with $K=Nk$ and $m_2$ which is of order $1/N$ by Laplace's method) we get that $e^{-N u}$ satisfies a LSI with constant $\kappa/N$ for some $\kappa>0$ independent from $N$.
\end{proof}

\section{Moments and propagation of chaos}

With the settings and notations of Section~\ref{sec:ovLangevin}, the purpose of this section is to gather several bounds on the process~\eqref{eq:overdampedUN} and the associated McKean-Vlasov diffusion
\begin{equation}
\label{eq:McKV}
\dd Y_t = -\na V(Y_t) \dd t - \na h_0(m_{\rho_t})\dd t  + \sqrt{2}\dd B_t \,,
\end{equation}
where $\rho$ solves~\eqref{eq:meanfieldEDP} (so that $Y_t \sim \rho_t$ for all $t\geqslant 0$). We work under the following conditions.

\begin{assu}
\label{assu:3}
 Assumption~\ref{assu:LSIN1} holds. Additionally,  $V$ satisfies all the conditions required for the confining potential in Assumption~\ref{assu:main-result}.
\end{assu}

\subsection{Moment bounds}\label{sec:moments}

Due to the super linear growth of $\na V$, the process comes back from infinity in finite time, which is manifested as follows.

\begin{lem}\label{lem:moments_instantanés}
Under Assumption~\ref{assu:3},  for any $t>0$ and $k,m\geqslant 2$,  there exists $R_{t,k,m}>0$ such that for all $N\geqslant 1$ and all initial condition $\bx\in\R^{dN}$, along~\eqref{eq:overdampedUN},
\[\mathbb E_{\bx} \po \frac{1}{N}\sum_{i=1}^N |X_t^i|^k\pf \leqslant R_{t,k,m}\,,\qquad \mathbb P_{\bx} \po \frac{1}{N}\sum_{i=1}^N |X_t^i|^k \geqslant R_{t,k,m} \pf\leqslant \frac{R_{t,k,m}}{N^m}\,.\]
\end{lem}
\begin{proof}
Denoting  $g_k(\bx)= \frac1N\sum_{i=1}^N |x_i|^k$ and $\mathcal L = -\na U_N \cdot \na + \Delta$ the generator of the process, we see that, for $k\geqslant 2$,
\[\mathcal Lg_k(\bx) \leqslant -\frac1{2}  g_{k-2+\beta}(\bx)  +C  \leqslant -\frac1{2}  g_{k}^{\frac{k-2+\beta}{k}}(\bx)  +C \]
for some $C>0$ independent from $t$ and $N$. Here we used that $\na h$ is Lipschitz and Taylor and Jensen inequality to bound
\[g_{k-1}(\bx) g_1(\bx) \leqslant g_k(\bx)  \leqslant \po g_{k-2+\beta}(\bx)\pf^{\frac{k}{k-2+\beta}} \leqslant \varepsilon g_{k-2+\beta}(\bx) + C_\varepsilon  \,,   \]
for any arbitrarily small $\varepsilon>0$, for some constant $C_\varepsilon>0$ depending only on $k,\beta$ and $\varepsilon$ but not $\bx$ nor $N$. By standard Lyapunov arguments with localization and Fatou Lemma we get that $\sup_{t\in[0,T]}\mathbb E(|\bX_t|^k) < \infty$ for all $T>0$, $k\in\N$.

Applying It$\bar{\text{o}}$ formula,
\begin{equation}
\label{loc:Ito}
\dd g_k(\bX_t) \leqslant - \frac1{2} g_{k}^{\frac{k-2+\beta}{k}}(\bX_t) \dd t + C \dd t + \frac{1}{\sqrt{N}} \dd M_t
\end{equation}
 where, thanks to the local time-uniform moment bound, $M$ is a martingale with quadratic variation
 \begin{equation}
 \label{loc:Mt}
 \dd [M_t] = k^2 g_{2(k-1)}(\bX_t)\dd t \,.
 \end{equation}
Setting $\alpha = \frac{k-2+\beta}{k}$, taking the expectation in the previous bound and using Jensen inequality,
\[\partial_t \mathbb E \po g_k(\bX_t)\pf  \leqslant - \frac12 \po \mathbb E \po g_k(\bX_t)\pf \pf^{\alpha}  + C \,,\]
so that
\begin{equation}
\label{loc:retour}
\mathbb E_{\bx} \po g_k(\bX_t)\pf  \leqslant (2C)^{1/\alpha} + \po \frac{\alpha-1}{4} t +    g_k^{1-\alpha}(\bx) \pf^{\frac1{1-\alpha}} \leqslant (2C)^{1/\alpha} + \po \frac{\alpha-1}{4} t\pf^{\frac1{1-\alpha}}\,,  
\end{equation}
which concludes the proof of the first point of the lemma.

For the second part, fixing a $t_0>0$, we will prove two things: first, that for all $k\geqslant 2$,  $g_k(\bX_t)$ goes down below some level $L_{(k)}$ (independent from $\bx$ and $N$) before time $t_0$ with high probability (as $N\rightarrow \infty$). Second, that, for all $k\geqslant 2$, starting from $\bx$ with $g_r(\bx)\leqslant L_{(r)}$ for some large $r$ (depending on $k$), then $g_k(\bX_t)$ does not go above a level $L_{(k)}'$ (for some $L_{(k)}'$ independent from $N$) during  time $t_0$ with high probability. Combining these two facts, conclusion then follows from the strong Markov  property.

For a fixed $k\geqslant 2$, writing $Z_t = 1+g_k(\bX_t)$ and using that $s^\alpha \geqslant \frac12 (1+s)^\alpha - C'$ for some constant $C'>0$  for all $s\geqslant 0$,  we get, for all $t\geqslant t_0/2$,
\begin{eqnarray*}
\dd Z_t^{1-\alpha} & = & (1-\alpha) Z_t^{-\alpha} \dd Z_t + \alpha (\alpha-1) Z_t^{-\alpha-1} \dd [M]_t \\
& \geqslant & (1-\alpha) Z_t^{-\alpha} \co \po - \frac14 Z_t^\alpha + C' + C + \frac{1}{\sqrt{N}}\dd M_t\pf \dd t  \cf \\
& = & \frac{\alpha-1}{4} - (\alpha-1) \frac{C'+C}{Z_t^\alpha} - \frac{1}{\sqrt{N}}\dd \tilde M_t
\end{eqnarray*}
with $\tilde M = (\alpha-1) \int_{t_0/2}^t Z_t^{1-\alpha}\dd  M_t$ a martingale with quadratic variation
\begin{equation}
\label{eq:tildemartingale}
\dd [\tilde M_t] = (1-\alpha)^2 Z_t^{-2\alpha}  k^2 g_{2(k-1)}(\bX_t)\dd t \leqslant (1-\alpha)^2  k^2 g_{2(k-1)}(\bX_t)\dd t \,.
\end{equation}
Fix $t_0>0$ and, for $L_0\geqslant 8(C'+C)$, consider the event
\[\mathcal A= \left\{\inf_{t\in[t_0/2,t_0]} Z_t^\alpha  \geqslant L_0\right\}\,.\]
Under this event,  integrating in time the previous inequality gives
\[Z_{t_0}^{1-\alpha} - Z_{t_0/2}^{1-\alpha} \geqslant \frac{(\alpha-1)t_0}{16} - \frac{1}{\sqrt{N}}\tilde M_{t_0}\,.\]
Since $Z_{t_0}^{1-\alpha} \leqslant L_0^{(1-\alpha)/\alpha}$ under $\mathcal A$, taking $L_0$ large enough so that $L_0^{(1-\alpha)/\alpha} \leqslant (\alpha-1)t_0/32=:q$,  we get, for any $p\geqslant 1$,
\[\mathbb P \po \mathcal A\pf \leqslant \mathbb P \po  \tilde M_{t_0}  \geqslant q \sqrt{N} \pf \leqslant q^{-p}N^{-p/2} \mathbb E \po  |\tilde M_{t_0}|^p\pf   \,. \]
By the Burkholder-Davis-Gundy inequality, using~\eqref{eq:tildemartingale} and the bounds on the moments $g_m(\bX_t)$ for any $m$ uniformly over $t\in[t_0/2,t_0]$ proven in the first part of the lemma, we get that for all $m>0$ there exists $C_m>0$ such that
\[\mathbb P_{\bx} \po \mathcal A\pf \leqslant \frac{C_m}{N^m}\,. \]
This shows that, with high probability (for large $N$), independently from the initial condition $\bx\in\R^{dN}$, $g_k(\bX_t)$ goes below the level $L_0^{1/\alpha}$ before time $t_0$.
 
 Next, for given $k,m\geqslant 2$, set $r=2^{2m}(k-1)^{2m}$ and let $L_{(r)},C_r,L_{(k)},C_k>0$ be such that for all $\bx\in\R^d$ and $N\geqslant 1$,
 \[\mathbb P_{\bx} \po \inf_{t\in[0,t_0]} g_r(\bX_t) \geqslant L_{(r)}\pf \leqslant \frac{C_r}{N^m}\,,\]
 and similarly for $L_{(k)},C_k$. Let $\bx\in \mathcal B:=\{\by\in\R^{dN},\ g_r(\by) \leqslant L_{(r)},\ g_k(\by) \leqslant L_{(k)}\}$. The first inequality of~\eqref{loc:retour} shows that
 \begin{equation}
\label{loc:supt}
\sup_{t\in[0,t_0]} \mathbb E\po g_r(\bX_t)\pf \leqslant L_{(r)} + (2C)^{1/\alpha}\,. 
 \end{equation}
Integrating~\eqref{loc:Ito} in time gives, for $t\in[0,t_0]$,
 \[g_k(\bX_t) \leqslant g_k(\bx) + C t + \frac{1}{\sqrt{N}} M_t \leqslant L_{(k)} + C t_0 + \frac1{\sqrt N} M_t\,. \]
 Hence,
 \[\mathbb P_{\bx} \po \sup_{t\in[0,t_0]} g_k(\bX_{t}) \geqslant L_{(k)} + C t_0 + 1 \pf \leqslant \mathbb P_{\bx} \po \sup_{t\in[0,t_0]} M_{t} \geqslant \sqrt{N}\pf \leqslant \frac{\mathbb E_{\bx} \po |M_{t_0}|^{2m}\pf}{N^m}\,.\]
Using again the Burkholder-Davis-Gundy inequality and bounding $\mathbb E_{\bx} \po |M_t|^{2m}\pf$ independently from $\bx$ and $N$ thanks to~\eqref{loc:Mt} and \eqref{loc:supt} with $r= 2^{2m}(k-1)^{2m}$, we get that 
 \begin{equation}\label{loc:lemmainegal}
 \sup_{\bx\in \mathcal B }\mathbb P_{\bx} \po \sup_{t\in[0,t_0]} g_k(\bX_{t}) \geqslant L_{(k)} + C t_0 + 1 \pf \leqslant \frac{C''}{N^k}
 \end{equation}
 for some $C''$ independent from $N$. By the strong Markov property, finally, 
 \[\mathbb{P}\po g_k(X_{t_0})  \geqslant L_{(k)} + C t_0 + 1 \pf \leqslant \frac{C_k + C_r + C''}{N^k}\,,\]
 which concludes the proof. 
 
 \end{proof}

\begin{lem}\label{lem:CVfast}
Under Assumption~\ref{assu:3}, assume moreover the uniform LSI~\eqref{unifLSI}. For all $r>0$, there exists $T>0$ such that for all $\rho_0 \in\mathcal P_2(\R^d)$, the solution of~\eqref{eq:meanfieldEDP} satisfies 
\begin{equation}
\label{eq:CVfast}
\mathcal W_2(\rho_T,\rho_*) \leqslant \frac12 \min \po r,  \mathcal W_2(\rho_0,\rho_*)\pf\,.
\end{equation}
\end{lem}
\begin{proof}
Reasoning as in Lemma~\ref{lem:moments_instantanés}, we see that, for any $k\geqslant 2$, $z_t  = \int_{\R^d} |x|^k \rho_t(\dd x)$ satisfies, 
\[\partial_t z_t \leqslant - \frac12 z_t^\alpha  + C \]
for some $\alpha>1$, leading to $z_t \leqslant M_t:= (2C)^{1/\alpha} + \po (\alpha-1) t/4\pf^{\frac1{1-\alpha}}$ independently from $z_0$. Applying this with $k=2$ and~\eqref{eq:W2contract}, for all $t\geqslant 1$,
\begin{eqnarray*}
\mathcal W_2^2(\rho_t,\rho_*) & \leqslant& C e^{-\lambda t} \min \co \mathcal W_2^2(\rho_0,\rho_*), e^{\lambda} \mathcal W_2^2(\rho_1,\rho_*)\cf \\
& \leqslant &   C e^{-\lambda t} \min \co \mathcal W_2^2(\rho_0,\rho_*), 2e^{\lambda}  \po \mathcal W_2^2(\delta_0,\rho_*) + M_1\pf \cf \,.
\end{eqnarray*}
Taking $t$ large enough concludes.
\end{proof}

\begin{lem}\label{lem:technique}
Under Assumption~\ref{assu:3}, for all $T,M>0$, $k\geqslant 2$ and  $r\geqslant 2^{2k}(\theta-1)^{2k}$, there exists $L,C>0$ such that for all $\rho_0\in\mathcal P_2(\R^d)$ with $\int_{\R^d} |x|^{r} \rho_0(\dd x) \leqslant M$ and all $N\geqslant 1$,
\[ \sup_{t\in[0,T]} \int_{\R^d} |x|^r \rho_0(\dd x) \leqslant C\,,\qquad   \mathbb P \po \sup_{t\in[0,T]} \frac1N \sum_{i=1}^N |Y_t^i|^{\theta} \geqslant L \pf \leqslant \frac{C}{N^k}\,,\]
where $Y^1,\dots Y^N$ are i.i.d. solutions of~\eqref{eq:McKV} with $\bY_0 \sim \rho_0^{\otimes N}$. 
\end{lem}
\begin{proof}
The proof is similar to Lemma~\ref{lem:moments_instantanés}, more precisely~\eqref{loc:retour} and~\eqref{loc:lemmainegal}, hence is omitted.
\end{proof}

\subsection{Finite-time trajectorial propagation of chaos}

\begin{lem}\label{lem:technique2}
Under Assumption~\ref{assu:3}, for all $T,M,\varepsilon>0$ and $k\geqslant 2$, there exists $C>0$  such that, writing $r = \max(k+2,  2^{2k}(\theta-1)^{2k})$, for all $\rho_0\in\mathcal P_2(\R^d)$ with $\int_{\R^d} |x|^{r} \rho_0(\dd x) \leqslant M$ and all $N\geqslant 1$, 
\[\mathbb P \po \sup_{t\in[0,T]} |\bar Y_t - m_{\rho_t}| \geqslant \varepsilon \pf \leqslant \frac{C}{N^k}\,,\]
where $\bar Y_t = \frac1N\sum_{i=1}^N Y_t^i$ with $\bY$  as in Lemma~\ref{lem:technique}.
\end{lem}
\begin{proof}
First, for $t \geqslant s\geqslant 0$,
\begin{equation}
\label{loc:mrho}
|m_{\rho_t}-m_{\rho_s}| = \left|\int_s^t \int_{\R^d}  \co \na V(x) + \na h_0(m_{\rho_u})\cf \rho_u(\dd x) \dd u \right| \leqslant C|t-s|
\end{equation}
for some constant $C$ thanks to Lemma~\ref{lem:technique},  using that $\na h$ is Lipschitz continuous and $|\na V(x)|\leqslant |x|^\theta + C'$ for some constant $C'$. Similarly, using that
\[\bar Y_t - \bar Y_s = \frac1N \sum_{i=1}^N \int_s^t \co \na V_0(Y_u^i) + \na h_0(m_{\rho_{u}})\cf \dd u + \frac1{\sqrt{N}}  (\bar B_t^i- \bar B_s^i)\,,\]
with $\bar B_t = \frac{1}{\sqrt{N}}\sum_{i=1}^N B_t^i$, using the polynomial bound on $\na V_0$ and Lemma~\ref{lem:technique} we get that there exist $L',C>0$ such that 
\[\mathbb P \po \sup_{t\in[0,T]} \frac1N \sum_{i=1}^N   | \na V_0(Y_t^i) + \na h_0(m_{\rho_{t}})| \geqslant L'  \pf \leqslant \frac{C}{N^k}\,.\]
Then, taking $p\in\N$ large enough so that $h:=T/p \leqslant \varepsilon/(8L')$,
\begin{multline*}
\mathbb P \po \exists  j \in \cco 0,p-1\ccf,\ \sup_{s \in [0,h]} |Y_{jh + s} - Y_{jh}| \geqslant \varepsilon /4  \pf\\
  \leqslant  \frac{C}{N^k} + \mathbb P \po \exists j \in \cco 0,p-1\ccf,\ \sup_{s \in [0,h]} |\bar B_{jh + s} - \bar B_{jh}| \geqslant \sqrt{N}\varepsilon/8 \pf \leqslant \frac{C'}{N^k}
\end{multline*}
for some constant $C'>0$. Moreover, in view of~\eqref{loc:mrho} we can additionally chose $p$ large enough so that
\[\sup_{t\in[0,T],s\in[0,h]}|m_{\rho_t}-m_{\rho_{t+s}}| \leqslant \varepsilon/4\,.\]
This leads to
\begin{eqnarray*}
\mathbb P \po \sup_{t\in[0,T]} |\bar Y_t - m_{\rho_t}| \geqslant \varepsilon \pf &\leqslant& \frac{C'}{N^k} + \mathbb P \po \max_{j\in\cco 0,p\ccf } |\bar Y_{jh} - m_{\rho_{jh}}| \geqslant \varepsilon/2 \pf\\
 &\leqslant& \frac{C'}{N^k} + \sum_{j=0}^p \mathbb P \po |\bar Y_{jh} - m_{\rho_{jh}}| \geqslant \varepsilon/2 \pf\,.
\end{eqnarray*}
Conclusion follows from \cite[Theorem 2]{FournierGuillin}, using the time-uniform moment bounds of Lemma~\ref{lem:technique} with $r \geqslant k+2$.
\end{proof}

\begin{lem}\label{lem:technique3}
Under Assumption~\ref{assu:3}, for all $T,M,\varepsilon>0$ and $k\geqslant 2$, there exists $C>0$  such that, writing $r = \max(2k+3,  2^{2k}(\theta-1)^{2k})$, for all $\rho_0\in\mathcal P_2(\R^d)$ with $\int_{\R^d} |x|^{r} \rho_0(\dd x) \leqslant M$ and all $N\geqslant 1$,
\[\mathbb P \po \sup_{t\in[0,T]} \mathcal W_2 (\pi(\bY_t), \rho_t) \geqslant \varepsilon \pf \leqslant \frac{C}{N^k}\,.\]
with $\bY$  as in Lemma~\ref{lem:technique}.
\end{lem}
\begin{proof}
For $s,t\geqslant 0$, we boud 
\[\frac1N\sum_{i=1}^N |Y_t^i - Y_{t+s} ^i|^2 \leqslant \frac{2}{N}\sum_{i=1}^N \left| \int_t^{t+s} \co \na V_0(Y_s^i) + \na h_0(m_{\rho_u})\cf \dd u \right|^2 +  \frac{2}{N}\sum_{i=1}^N |B_{t+s}^i - B_t^i|^2\,.\]
Reasoning as in Lemma~\ref{lem:technique2}, we get that  there exists $h,C>0$ such that
\[ \sup_{t\in[0,T]} \mathbb P \po \sup_{s\in[0,h]} \frac1N\sum_{i=1}^N |Y_{t+s}^i - Y_t^i|^2 \geqslant \varepsilon^2/16 \pf \leqslant \frac{C}{N^k}\,,\]
which implies
 \[ \sup_{t\in[0,T]} \mathbb P \po \sup_{s\in[0,h]} \mathcal W_2\po \pi(\bY_t),\pi(\bY_{t+s})\pf  \geqslant \varepsilon/4 \pf \leqslant \frac{C}{N^k}\,.\]
 Similarly, using that $\mathcal W_2^2(\rho_{t+s},\rho_t) \leqslant \mathbb E \po |Y_{t+s}^1 -Y_t^1|^2\pf \leqslant C \sqrt{|t-s|}$, we can take $h$ small enough so that $\mathcal W_2^2(\rho_{t+s},\rho_t)  \leqslant \varepsilon/4$ for all $s\in[0,h]$ and $t\in[0,T]$. As in Lemma~\ref{lem:technique2}, we get that 
 \[
 \mathbb P \po \sup_{t\in[0,T]} \mathcal W_2 (\pi(\bY_t), \rho_t) \geqslant \varepsilon \pf  \leqslant \frac{C \lceil T/h\rceil}{N^k} + \sum_{j=0}^{\lceil T/h\rceil} \mathbb P \po \mathcal W_2\po \pi(\bY_{jh}),\rho_{jh}\pf  \geqslant \varepsilon/2 \pf  
 \]
 and conclude with \cite[Theorem 2]{FournierGuillin} (with a bound on the moment of order $r\geqslant 2(k+1)+1$).
\end{proof}

\begin{prop}\label{prop:PoCrhoX}
Under Assumption~\ref{assu:3}, for all $T,M,\varepsilon>0$ and $k\geqslant 2$, there exists $C>0$  such that, writing $r = \max(2k+3,  2^{2k}(\theta-1)^{2k})$, for all $\rho_0\in\mathcal P_2(\R^d)$ with $\int_{\R^d} |x|^{r} \rho_0(\dd x) \leqslant M$ and all $N\geqslant 1$, for all $\bx\in\R^{dN}$ with $\frac1N\sum_{i=1}^N |x_i|^r \leqslant M$,
\[\mathbb P\po \sup_{t\in[0,T]} \mathcal W_2\po \pi(\bX_t) , \rho_t^{\bx}\pf \geqslant \varepsilon \pf \leqslant \frac{C}{N^k}\,,\]
where $\bX$ solves~\eqref{eq:overdampedUN} and $\rho^{\bx}$ is the solution of \eqref{eq:meanfieldEDP} with $\rho_0 = \pi(\bx)$. 
\end{prop}

\begin{proof}
Let $\tilde \bY_0 \sim \pi(\bx)^{\otimes N}$. Let $\sigma$ be a permutation of $\cco 1,N\ccf$ such that
\[\mathcal W_2^2\po \pi(\bx),\pi\po \tilde \bY_0\pf\pf = \frac1N \sum_{i=1}^N |x_i - \tilde Y_0^{\sigma(i)}|^2\,.\]
Then $\bY_0:=(\tilde Y_0^{\sigma(1)},\dots,\tilde Y_0^{\sigma(N)}) \sim \pi(\bx)^{\otimes N}$ and $\mathcal W_2\po \pi(\bx),\pi(\bY_0)\pf = |x-\bY_0|/\sqrt{N}$. Let $\bY_t$ solve
\[\dd Y_t^i = -\na V_0(Y_t^i) \dd t - \na h(m_{\rho_t^{\bx}}) \dd t + \sqrt{2}\dd B_t^i\,,\]
with the same Brownian motions as $\bX$. Then, using that $\na h$ is Lipschitz and $V$ is one-sided Lipschitz, we get a constant $L'>0$ such that 
\begin{align*}
\dd |X_t^i - Y_t^i|^2 &= 2 \po X_t^i-Y_t^i\pf \cdot \po \na V(Y_t^i)-\na V(X_t^i) + \na h(m_{\rho_t^{\bx}}) - \na h\po \bar X_t\pf  \pf \dd t \\
&\leqslant  L' \po |X_t^i - Y_t^i|^2 + \frac1N \sum_{j=1}^N |X_t^j - Y_t^j|^2 \pf\dd t  + \|\na^2 h\|_\infty |\bar Y_t - m_{\rho_t^{\bx}}|^2 \dd t 
\end{align*}

For $\varepsilon>0$, consider the event
\[A_N^\varepsilon = \left\{ \sup_{t\in[0,T]} \mathcal W_2\po  \pi(\bY_t) ,\rho_t^{\bx}\pf  \leqslant \varepsilon\right\}\,.\]
Under this event, the previous inequality and Gronwall lemma gives, for all $t\in[0,T]$,
\begin{align*}
\mathcal W_2^2\po \pi(\bX_t),\rho_t^{\bx}\pf & \leqslant 2\mathcal W_2^2\po \pi(\bX_t),\pi(\bY_t)\pf + 2\mathcal W_2^2\po \pi(\bY_t),\rho_t^{\bx}\pf  \\
&\leqslant \frac{2}{N}\sum_{i=1}^N |X_t^i - Y_t^i|^2 +   2\varepsilon^2   \\
&\leqslant 2\po e^{L'T}\co T\varepsilon^2 \|\na^2 h\|_\infty +  \frac{1}{N}\sum_{i=1}^N |x_i - Y_0^i|^2\cf\pf  +   2\varepsilon^2\\
&\leqslant 2\po e^{L'T}\co T \|\na^2 h\|_\infty +  1\cf\pf \varepsilon^2  +   2\varepsilon^2 \,,
\end{align*}
where we used that under $A_N^\varepsilon$,
\[ \frac{1}{N}\sum_{i=1}^N |x_i - Y_0^i|^2 = \mathcal W_2^2 \po \pi(\bY_0),\pi(\bx)\pf = \mathcal W_2^2 \po \pi(\bY_0),\rho_0^{\bx}\pf  \leqslant \varepsilon^2\]
As a consequence, for any fixed $\varepsilon,T>0$, we take $\varepsilon'>0$ small enough so that
\[A_N^{\varepsilon'} \subset\left\{ \sup_{t\in[0,T]} \mathcal W_2\po \pi(\bX_t) , \rho_t^{\bx}\pf < \varepsilon \right\}\,.\]
Conclusion follows Lemma~\ref{lem:technique3}.
\end{proof}

\begin{prop}\label{prop:finiteTimePOC}
Under Assumption~\ref{assu:3}, for all $T,M,\varepsilon>0$ and $k\geqslant 2$, there exists $C>0$  such that, writing $r = \max(2k+3,  2^{2k}(\theta-1)^{2k})$, for all $\rho_0\in\mathcal P_2(\R^d)$ with $\int_{\R^d} |x|^{r} \rho_0(\dd x) \leqslant M$ and all $N\geqslant 1$, 
\[\mathbb P  \po \sup_{t\in[0,T]} \mathcal W_2\po \pi(\bX_t) , \rho_t\pf \geqslant \varepsilon \pf \leqslant \frac{C}{N^k}\,,\]
where $\bX$ solves~\eqref{eq:overdampedUN} with $\bX_0\sim \rho_0^{\otimes N}$  and $\rho$ is the solution of \eqref{eq:meanfieldEDP}. 
 
\end{prop}

\begin{proof}
The proof is similar to Lemma~\ref{prop:PoCrhoX}, taking into account that, at the initial time, for any $\varepsilon'>0$,
\[\mathbb P  \po   \mathcal W_2\po \pi(\bX_0) , \rho_t\pf \geqslant \varepsilon' \pf = \underset{N\rightarrow \infty}{\mathcal O}\po N^{ -k}\pf \,,\]
thanks to \cite[Theorem 2]{FournierGuillin} and the moment assumption on $\rho_0$.
\end{proof}

\end{document}